\documentclass{amsart}

\usepackage{url}

\usepackage{color}
\definecolor{darkgreen}{rgb}{0,0.5,0}
\usepackage[
        colorlinks, citecolor=darkgreen,
        backref,
        pdfauthor={Jennifer S. Balakrishnan, Amnon Besser, Francesca Bianchi, J. Steffen Mueller},
        pdftitle={Explicit quadratic Chabauty over number fields}
]{hyperref}
\usepackage{amssymb, amsmath, amsthm}
\usepackage{colonequals}
\usepackage{enumitem}
\usepackage{xypic, comment,braket}
\usepackage{pdflscape}
\usepackage{afterpage, multirow}

\newcommand{\rank}{\textrm{rank}}
\newcommand{\atp}{f_{2g+1}}

\newcommand{\Q}{\mathbb{Q}}
\newcommand{\Z}{\mathbb{Z}}
\newcommand{\A}{\mathbb{A}}
\newcommand{\PP}{\mathbb{P}}

\DeclareMathOperator{\Res}{Res}
\DeclareMathOperator{\Lie}{Lie}
\newcommand{\Qp}{\Q_p}
\newcommand{\zvec}{\underline{z}}

\newcommand{\F}{\mathbb{F}}
\newcommand{\fp}{\mathfrak{p}}
\newcommand{\fq}{\mathfrak{q}}

\newcommand{\Qpb}{\bar{\Q}_p}

\newcommand{\dr}{\textup{dR}}
\newcommand{\hdr}{H_{\dr}}
\newcommand{\DD}{\mathcal{D}}

\newcommand{\OO}{\mathcal{O}}

\newcommand{\lc}{\chi^{\textup{cyc}}}
\newcommand{\cyc}{\textup{cyc}}
\newcommand{\hc}{h^{\textup{cyc}}}
\newcommand{\rc}{\rho^{\textup{cyc}}}
\newcommand{\la}{\chi^{\textup{anti}}}
\newcommand{\ha}{h^{\textup{anti}}}
\newcommand{\anti}{\textup{anti}}
\newcommand{\ra}{\rho^{\textup{anti}}}

\newcommand{\XX}{\mathcal{X}}
\newcommand{\UU}{\mathcal{U}}

\newcommand{\bom}{\bar{\omega}}
\newcommand{\eps}{\varepsilon}

\newcommand{\ord}{\operatorname{ord}}
\newcommand{\ol}[1]{\overline{#1}}

\newcommand{\Hom}{\operatorname{Hom}}
\newcommand{\Jac}{\operatorname{Jac}}

\newcommand{\Div}{\operatorname{Div}}
\newcommand{\Bil}{\mathcal{B}}
\newcommand{\Spec}{\operatorname{Spec}}
\newcommand{\supp}{\operatorname{supp}}

\newcommand{\tr}{\operatorname{tr}}
\newcommand{\rk}{\operatorname{rk}}

\newcommand{\ul}[1]{\underline{#1}}

\newtheorem{theorem}{Theorem}[section]

\newtheorem{lemma}[theorem]{Lemma}
\newtheorem{corollary}[theorem]{Corollary}
\theoremstyle{definition}
\newtheorem{definition}[theorem]{Definition}
\newtheorem{remark}[theorem]{Remark}
\newtheorem{example}[theorem]{Example}
\newtheorem{assumption}[theorem]{Assumption}
\newtheorem{condition}[theorem]{Condition}
\newtheorem{question}[theorem]{Question}
\numberwithin{equation}{section}
\newtheorem{algorithm}{Algorithm}
\newtheorem{claim}{Claim}

\renewcommand*{\backref}[1]{}
\renewcommand*{\backrefalt}[4]{%
  \ifcase #1 %
    \relax
  \or
    $\uparrow$#2.%
  \else
    $\uparrow$#2.%
  \fi%
}

  \renewenvironment{thebibliography}[1]{%
    \begin{oldthebibliography}{#1}%
      \setlength{\parskip}{0ex}%
      \setlength{\itemsep}{0.1ex}%
      \setlength{\labelwidth}{1.4cm} 
  }%
  {%
    \end{oldthebibliography}%
  }

\begin{document}
\title{Explicit quadratic Chabauty over number fields}
\author[J.S. Balakrishnan]{Jennifer S. Balakrishnan}
\address{Jennifer S. Balakrishnan, Department of Mathematics and Statistics, Boston University, 111 Cummington Mall, Boston, MA 02215, USA}

\author[A. Besser]{Amnon Besser}
\address{Amnon Besser, Department of Mathematics\\Ben-Gurion University of the Negev\\P.O.B. 653\\Be'er-Sheva 84105\\Israel}

\author[F. Bianchi]{Francesca Bianchi}
\address{Francesca Bianchi,  Bernoulli Institute for Mathematics, Computer Science and Artificial Intelligence\\ University of Groningen, Groningen, The Netherlands}

\author[J.S. M\"uller]{J. Steffen M\"{u}ller}
\address{J. Steffen M\"{u}ller, Bernoulli Institute for Mathematics, 
Computer Science and Artificial Intelligence\\ University of Groningen, Groningen, The Netherlands}

\begin{abstract}
  We generalize the explicit quadratic Chabauty techniques for integral points on odd degree hyperelliptic curves
  and for rational points on genus 2 bielliptic curves to arbitrary number fields using
  restriction of scalars.
  This is achieved by combining equations coming from Siksek's extension of classical Chabauty with equations defined in terms of $p$-adic heights attached to independent continuous idele class characters.
  We give several examples to show the practicality of our methods.
\end{abstract}

\date{\today}

\maketitle
\tableofcontents
\section{Introduction}\label{S:intro}

Let $K$ be a number field and let $X/K$ be a smooth projective curve of genus $g\ge 2$.
Then, by Faltings' theorem, the set of rational points $X(K)$ is finite, but at present no
general algorithm for the computation of $X(K)$ is known. 
When the Mordell-Weil rank $r$
of the Jacobian $J/K$ is less than $g$, the method of Chabauty~\cite{Chab41}, made effective by
Coleman~\cite{Col85, mccallum-poonen:method}, can often be used to compute a finite subset of $X(K_\fp)$ containing $X(K)$, for any prime $\fp$ of good reduction for $X$.
 Combined with other techniques such as the Mordell-Weil
sieve~\cite{Bruin-Stoll:MWSieve}, this suffices to compute $X(K)$ exactly in many situations.
In principle, this method can be used also to compute rational points on elliptic curves, but it only applies to the trivial case $r=0$. 

The method of Chabauty-Coleman relies on the image of the map
\begin{equation}\label{log}
  \log\colon J(K)\otimes K_\fp\to H^0(X_{K_\fp}, \Omega_1)^\vee
\end{equation}
having positive codimension, which is used to write down abelian integrals vanishing at
the rational points.
Kim~\cite{Kim05, Kim09} proposed to extend this approach to curves with $r\geq g$, by replacing the Jacobian
with non-abelian Selmer varieties and the abelian integrals with iterated Coleman integrals. 
This requires techniques from $p$-adic Hodge theory and, in general, appears to be quite difficult to use for computations of $X(K)$.
However, the simplest non-abelian instance of Kim's program has been made
explicit in various circumstances.

For now suppose that $K=\Q$.
The basic idea is to use a $p$-adic height function, which is a quadratic form
$h\colon J(\Q)\to \Q_p$. We can decompose $h$
into a sum $h = \sum_{v} h_v$ of local terms $h_v\colon
\Div^0(X_{\Q_v})\to \Q_p$, where $v$ runs through all prime numbers and $h_v(D)=0$ for all but finitely many $v$ if $D\in \Div^0(X_{\Q})$ (see~\cite{Col-Gro89,
BB:Crelle}).
If~\eqref{log} is injective, then 
we can write down a
basis $\mathcal{Q}$ of the quadratic forms on $J(\Q)\otimes \Q_p$ in terms of products of abelian
integrals, and we can write $h$ as a linear
combination $h=\sum_{q \in \mathcal{Q}} \alpha_q q$ over $\Q_p$. Hence we find
\begin{equation}\label{}
  \sum_{q \in \mathcal{Q}} \alpha_q q -  h_p = \sum_{v \ne p} h_v.
\end{equation}
The left hand side can be pulled back to $X(\Q)$ and extended to a locally analytic function
on a subset $\mathcal{S}_p\subset X(\Q_p)$. If the pull-back of the right hand side vanishes on $\prod _{v\neq p}\mathcal{S}_v \cap X(\Q)$ for some $\mathcal{S}_v\subset X(\Q_v)$ or, more generally, if we can control
the image, then  
we get a locally analytic function on $\mathcal{S}_p$ which takes values on $\prod_v\mathcal{S}_v\cap X(\Q)$ in an explicitly computable finite subset of $\Q_p$.

One context in which this idea has been made explicit is the one of integral points
on certain hyperelliptic curves.
Suppose that $X$
is hyperelliptic, given by a model $y^2=f(x)$, where $f\in \Z[x]$ is monic of odd degree
and has no repeated roots.  
It is shown
in~\cite{BBM14} that 
\begin{enumerate}
  \item the function $\tau(z)\colonequals h_p(z-\infty)$ is a Coleman function on
    $X(\Q_p)\setminus\{\infty\}$, which can
    be expressed as a double Coleman integral (\cite[Theorem~2.2]{BBM14});
  \item for $v\ne p$, the function $z \mapsto h_v(z-\infty)$ takes values on the
    $v$-integral points of $X(\Q_v)$ in an
    explicitly computable finite set $T_v\subset \Q_p$ (\cite[Proposition~3.3]{BBM14}), which is the set $\{0\}$  for almost all $v$.
\end{enumerate}
These results do not use Kim's approach directly; the first uses $p$-adic Arakelov
theory whereas the second is based on intersection theory on arithmetic surfaces.
Let $\UU=\Spec \Z[x, y] / (y^2 - f(x))$. 
Together with the discussion above, it is then easy to deduce
\begin{theorem}[{\cite[Theorem~3.1]{BBM14}}]\label{T:QC0}
  Suppose that $r=g$ and that the map in \eqref{log} is an isomorphism.
  Then there exists an explicitly computable finite set $T\subset \Q_p$  and an explicitly
  computable non-constant Coleman function $\rho\colon \UU(\Z_p) \to \Q_p$ such that
$ \rho(\UU(\Z)) \subset T$. 
\end{theorem}
Here and in the following we write that certain $p$-adic objects are {\em explicitly
computable} if we can provably compute them to any desired $p$-adic precision;
see Remark~\ref{R:prec_ana}.

The first goal of the present article is to extend Theorem~\ref{T:QC0}  to general number fields.
Fix a prime number $p$ such that $X$ has good reduction at all $\fp\mid p$ and
fix a nontrivial continuous idele class character $\chi \colon  \mathbb{A}_K^\ast/K^\ast \to \Q_p$.
This choice induces $\Q_p$-linear maps $t^\chi_{\fp}\colon K_\fp\to \Q_p$ for all $\fp \mid p$ and a
$p$-adic height pairing $h^\chi$ with values in $\Q_p$. This can be decomposed into a sum of local
height pairings and
the local heights $h_{\fp}^{\chi}$ are of the form $t_{_\fp}^{\chi}\circ \tau_{\fp}$ for some Coleman functions $\tau_{\fp}$.
Following a similar extension of classical Chabauty due to Siksek~\cite{siksek:ecnf}, we work in $X(K\otimes
\Q_p)$ rather than in $X(K_\fp)$ for a single place $\fp \mid p$, and we consider the
composition
\begin{equation}\label{log2}
  \log\colon J(K)\otimes \Q_p\to \Res_{K/\Q}(J)(\Q)\otimes \Q_p \to \Lie(\Res_{K/\Q}(J))_{\Q_p}\,.
\end{equation}
Let $\sigma\colon X(K)
\hookrightarrow X(K\otimes \Q_p)$ denote the embedding induced by the product of the completion maps $K\hookrightarrow K\otimes \Q_p=\prod_{\fp \mid p} K_{\fp}$.

\begin{theorem}\label{T:int_intro}
  Suppose that \eqref{log2} is injective. Then there exists an explicitly computable
  finite set $T^{\chi}\subset \Q_p$  and an explicitly computable non-constant locally analytic function
  $\rho^\chi \colon \UU(\OO_K\otimes\Z_p)\to \Q_p$, both dependent on $\chi$, such that $\rho^\chi(\sigma(\UU(\OO_K))) \subset T^{\chi}$.
\end{theorem}
\begin{remark}
One can obtain a slight practical improvement of the theorem by noting that the sets $T^{\chi}$ arise, just as in~\cite{BBM14}, from possible values of the local components $h^\chi_v$, and these values are highly dependent for varying $\chi$ because they are just the products of a certain intersection pairing with a constant depending on $\chi_v$ (see~\eqref{away_disjoint}). Thus, picking a basis $\{\chi_i\}$ for the space of idele class characters one may prove that the vector valued function $ \ul{\rho}=(\rho^{\chi_1},\ldots)$ takes on $\sigma(\UU(\OO_K))$ a finite computable set of possible values, which is smaller than the obvious product of the $T^{\chi_i}$. For simplicity we ignore this point here. 
\end{remark}

See Theorem~\ref{T:rhos} below and its proof.
Note that both Theorem~\ref{T:QC0} and Theorem~\ref{T:int_intro} hold true when $g=1$.

The method of \cite{BBM14} was generalized by the first author and Dogra in \cite{BD18} to rational points on
smooth projective curves, satisfying some additional conditions. In contrast to the proof
of Theorem~\ref{T:QC0} in~\cite{BBM14}, this generalization uses Kim's approach directly, relating certain Selmer varieties to $p$-adic heights as constructed by
Nekov\'a\v r ~\cite{Nek93}, via $p$-adic Hodge theory.

In a special case, this was turned into an explicit algorithm to compute a finite set of $p$-adic points containing the rational points of $X$. 
In particular, let $X$ be a bielliptic curve of genus $2$ over a number field $K$. Then
$X$ admits degree $2$ maps $\varphi_1,\varphi_2$ to two elliptic curves $E_1$ and $E_2$, respectively. 

\begin{theorem}[{\cite[Theorem~1.4]{BD18}}]
\label{thm:BDQC1}
Suppose that $K$ is $\Q$ or an imaginary quadratic field in which $p$ splits, and that
  $E_1$ and $E_2$ each have rank $1$ over $K$; let $\fp\mid p$. Then there exist a nontrivial continuous idele class character $\chi$, an explicitly computable finite set $T^{\chi}\subset \Q_p$ and an explicitly computable non-constant Coleman function $\rho^{\chi}\colon X(K_{\fp})\to\Q_p$, such that $\rho^{\chi}(X(K))\subset T^{\chi}$.
\end{theorem}
This explicit result from \cite{BD18} can alternatively be proved only from
properties of the local heights on the elliptic curves $E_i$ and Theorem \ref{T:QC0} for
each $E_i$. Therefore, our approach to Theorem \ref{T:int_intro} can be used to extend this elementary proof of Theorem \ref{thm:BDQC1} to general number fields $K$.
\begin{theorem}\label{biell_intro}
Suppose that \eqref{log2} is an injection for each $E_i$; let $\mathcal{Q}_i$ denote a set of $\Q_p$-valued functions on $E_i(K\otimes \Q_p)$ which restricts to a basis of quadratic forms on $E_i(K)$. For every nontrivial continuous idele class character $\chi$, there exist explicitly computable constants $\alpha_{q_1}^{\chi},\beta_{q_2}^{\chi}\in\Q_p$ and an explicitly computable finite set $T^{\chi}\subset \Q_p$ such that the function
\begin{equation*}
\rho^{\chi}\colon \tilde{X}(K\otimes \Q_p)\to \Q_p
\end{equation*}
defined by
\begin{align*}
\rho^{\chi}(\zvec)\colonequals &\sum_{\fp\mid p}(h_{\fp}^{\chi}(\varphi_1(z_{\fp}))-h_{\fp}^{\chi}(\varphi_2(z_{\fp}))-2\chi_{\fp}(x(z_{\fp})))\\
- &\sum_{q_1 \in \mathcal{Q}_1}\alpha_{q_1}^\chi q_1(\varphi_1(\zvec))+\sum_{q_2 \in \mathcal{Q}_2}\beta_{q_2}^\chi q_2(\varphi_2(\zvec))
\end{align*}
satisfies
\begin{equation*}
\rho^{\chi}(\sigma(\tilde{X}(K)))\subset T^{\chi};
\end{equation*}
here $\tilde{X}(K\otimes \Q_p)$ is the subset of $X(K\otimes \Q_p)$ where $\rho^{\chi}$ is well-defined and $\sigma(\tilde{X}(K))$ is its intersection with $\sigma(X(K))$.
\end{theorem}

See Theorem~\ref{thm:bielliptic} for a more precise formulation. If we want to use Theorem~\ref{T:int_intro} (respectively Theorem~\ref{biell_intro}) to actually compute 
integral (respectively rational) points, then we need enough functions $\rho^{\chi}$ so
that their common zero set is finite. 
In order to achieve this, we require at least $[K:\Q] = \dim\Res_{K/\Q}X$ such functions.

Over the rational numbers, the space of continuous $\Q_p$-valued idele class characters has
dimension 1, so up to a scalar factor, Theorem~\ref{T:int_intro} (respectively Theorem~\ref{biell_intro}) only leads to one locally analytic
function which vanishes at the integral (respectively rational) points. 
In general, the dimension of this space is at least $r_2+1$, where $r_2$ is the number of
conjugate pairs of non-real embeddings of $K$ into $\mathbb{C}$ (with equality if Leopoldt's
conjecture holds). Hence we can expect $r_2+1$ independent functions.
Combining our functions with Siksek's work, we expect that generically, we get a method
to compute integral or rational points when
\begin{equation}\label{ranks_fin}
  r +\rk(\OO_K^\times) \le [K:\Q]\cdot g\,.
\end{equation}
As in Siksek's work, our approach will usually fail if $X$ can be defined over a subfield $F$ of $K$ and if  
\[
  \rk (\Jac(X)/F) +\rk(\OO_F^{\times}) > [F:\Q]\cdot g.
\]

It would be
interesting to investigate conditions which guarantee that our functions cut out a finite
set. See recent
work of Dogra~\cite{Dog19} and of Hast~\cite{Has19} for discussions of finiteness issues related to ours, but
working directly with Kim's approach.

However, the focus of this article is on developing methods for {\em computing} integral or
rational points, not on theoretical conditions guaranteeing finiteness. 
To illustrate the practicality of our method, we give several examples \footnote{
We focus on examples over quadratic fields because some of the restrictions of our method
are much easier to satisfy for such fields than for a higher degree number field, see
Remark~\ref{R:only_quadratic}.}
, combining our
techniques with the Mordell-Weil sieve following~\cite{BBM16} and~\cite[Appendix~B]{BD18}.
In particular, we compute 
\begin{itemize}
  \item the $\OO_{\Q(\sqrt{-3})}$-integral points on the genus 2 curve 
    defined by $y^2 = x^5 - x^4 + x^3 + x^2 - 2x + 1$, see Example~\ref{int_ex}. Its Jacobian
    is an optimal quotient of $J_0(188)$ and has rank~2 over $\Q$ and rank~4 over
    $\Q(\sqrt{-3})$. 
  \item the $\Q(i)$-rational points on the bielliptic modular curve $X_0(91)^+$, see
    Example~\ref{91_ex}. Note that Theorem \ref{thm:BDQC1} does not apply here, since the
    Mordell-Weil rank of the Jacobian of $X_0(91)^+$ over $\Q(i)$ is $4$. Its rank over
    $\Q$ is~2.
  \item the $\Q(\sqrt{34})$-rational points on the bielliptic curve $X\colon
    y^2=x^6+x^2+1$, see Example~\ref{Wetherell}. Inspired by a problem in Diophantus'
    {\em Arithmetica}, Wetherell determined the rational points on $X$ in his
    thesis~\cite{Wet97}. The rank of the Jacobian of $X$ over $\Q$ is 2 and the rank over $\Q(\sqrt{34})$ is~3.
\end{itemize}
It is immediate that Siksek's method does not apply to these examples because the
respective ranks are too large. They might be amenable to an approach via elliptic curve
Chabauty as in \cite{Bru03} (the relevant Galois group in the first example is of order~10), but to our knowledge the only
existing implementations of this method require the base field to be~$\Q$. 
We are also not aware of any method for computing integral points on hyperelliptic curves
over number fields that could have succeeded in the first example.

The outline of the paper is as follows: After setting some notation, we give an explicit
description of continuous idele class characters and the $p$-adic heights associated to
them in Section~\ref{S:hts}. We then recall Siksek's extension of Chabauty-Coleman in
Section~\ref{S:cnf}. In Section~\ref{S:qcnf} we extend the quadratic Chabauty techniques for
integral points on odd degree hyperelliptic curves
introduced in~\cite{BBM14} and prove Theorem~\ref{T:int_intro}. Similarly,
Section~\ref{sec:rat} is devoted to an extension of the explicit methods for bielliptic
curves of genus~2 from~\cite{BD18} and contains a proof of Theorem~\ref{biell_intro}.
We then turn to algorithmic details of these methods and give several examples for curves
of small genus defined over quadratic fields in
Sections~\ref{S:alg_ex_int} and~\ref{S:examples_rational}.
Finally, we discuss how we can provably find the solutions of the power series
equations resulting from our methods in Appendix~\ref{appendix:roots_multi}.

It would be interesting to extend the explicit results in~\cite{BD17} and~\cite{BDMTV} to
general number fields. We have not attempted this, but rather focused on extending those 
quadratic Chabauty results which do not require $p$-adic Hodge theory.

\section*{Acknowledgements}
We thank Keith Conrad for useful correspondence about his notes \cite{conrad:multihensel}
and for consequently including a power series version of the multivariate Hensel's Lemma, which was needed for
this work. 
We also thank the referee for several helpful remarks.
Part of the work described in this article was carried out during a visit of the first and the
fourth author to Ben-Gurion University. We would like to thank the Center for
Advanced Studies in Mathematics at Ben-Gurion University for their hospitality and support. 
The first author is supported in part by NSF grant DMS-1702196, the Clare Boothe Luce Professorship (Henry Luce Foundation), and
Simons Foundation grant \#550023. The third author was supported by EPSRC and by Balliol College through a Balliol Dervorguilla scholarship.
The third and fourth author acknowledge support from NWO through a VIDI grant.

\subsection{Notation}\label{S:notn}
We fix, once and for all, the following notation.
\[\begin{array}{ll}
    K & \textrm{a number field}\,,\\
  \OO_K& \textrm{the ring of integers of }K\,,\\
     d& \textrm{the extension degree }[K:\Q]\,,\\
    (r_1,r_2)&\textrm{the signature of }K\,,\\
         r_K & \textrm{the rank }  r_1+r_2 -1 \textrm{ of the unit group }\OO_K^\ast\,,\\
                     h_K & \textrm{the class number of }K\,,\\
                     K_v & \textrm{the completion of }K\textrm{ at a place }v\,,\\
                       p & \textrm{an odd prime number, unramified in }K\,,\\
    \fp_1,\ldots,\fp_m\subset \OO_K & \textrm{the primes of }\OO_K\textrm{ lying above } p\,,\\
  \sigma_j\colon K \hookrightarrow K_j \colonequals K_{\fp_j} & \textrm{the completion of }K\textrm{ at }\fp_j\,,\\
   \OO_j            & \textrm{the ring of integers of }K_j\,,\\
    \xi_\fq & \textrm{a generator of } \fq^{h_K} \textrm{ if } \fq\subset K \textrm{
  is a prime ideal}\,.\\
(K\otimes \Q_p)^\vee  & \prod^m_{j=1}\Hom_{\Q_p}(K_{j},\Q_p).
\end{array}\]

\section{$p$-adic heights}\label{S:hts}
\subsection{Continuous idele class characters}\label{S:icc}
In this section, we recall continuous idele class characters 
 \begin{equation*}
   \chi = \sum_v \chi_v \colon  \mathbb{A}_K^\ast/K^\ast \to \Q_p\;
 \end{equation*}
 and discuss how to construct them explicitly.
These characters are of great importance in Iwasawa theory, as they correspond to
$\Z_p$-extensions of $K$. We shall not use this interpretation; 
for us, the main application is the explicit construction of $p$-adic
heights in \S\ref{S:padic_hts}.

First we note the following:
\begin{itemize}
\item For any prime $\fq\nmid p$ we have $\chi_\fq(\OO_{K_\fq}^\ast)=0$ because of continuity.
  So if $\pi_\fq$ is a uniformizer in $K_\fq$, then $\chi_\fq$ is completely
  determined by $\chi_\fq(\pi_\fq)$. Equivalently,
  we can use $\chi_\fq(q)$, where $q$ is the underlying prime number, or $\chi_{\fq}(\xi_\fq)$.
\item For any $j \in \{1,\ldots,m\}$, there is a $\Q_p$-linear map $t^{\chi}_j$ such that we can decompose
  \begin{equation}\label{tdefnd}
    \xymatrix{
      {\OO_{j}^\ast}  \ar[rr]^{\chi_{\fp_j}} \ar[dr]^{\log_j} & &   \Q_p,\\
      & K_j\ar[ur]^{t^{\chi}_j}
  }
  \end{equation}
  because $\chi_{\fp_j}$ takes values in the torsion-free group $(\Q_p,+)$. We call the
    $t^{\chi}_j$ {\em trace maps}.
\end{itemize}

 \begin{remark}\label{ramified}
If a continuous idele class character $\chi$
   is ramified at $\fp_j$, that is, if the local
 character $\chi_{\fp_j}$ does not vanish on  $\OO^\ast_j$,
 then we can extend $\log_j$ to
  \begin{equation}\label{logbranch}
  \log_j  \colon  K_j^\ast \to K_j
\end{equation}
 in such a way that the diagram~\eqref{tdefnd} remains commutative.
 \end{remark}
 
By the above, a continuous idele class character of $K$ determines an element $t^\chi$ of
$(K\otimes\Q_p)^\vee$. We now show that the converse also holds.

\begin{lemma}\label{L:chi_t}
  A continuous idele class character $\chi \colon  \mathbb{A}_K^\ast/K^\ast \to \Q_p$ is
  uniquely determined by 
  \[
    t^\chi = (t^{\chi}_{1},\ldots,t^{\chi}_{m}) \in (K\otimes \Q_p)^\vee  
\] 
  where the $t^{\chi}_j$ are as in~\eqref{tdefnd}.
\end{lemma}
\begin{proof}
  We first show that $t^\chi$ determines $\chi_{\fp_j}$ for $1 \le j \le m$.
  Indeed, 
  since $\sigma_i(\xi_{\fp_j})$ is a unit in $\OO_{i}$ for $i \ne j$, we have
  \[
    \chi_{\fp_j}(\sigma_j(\xi_{\fp_j})) = -\sum_{i\ne
    j}\chi_{\fp_i}(\sigma_i(\xi_{\fp_j})) =
    -\sum_{i\ne j}t^{\chi}_{i}(\log_i(\sigma_i(\xi_{\fp_j})))\,,
  \]
  so $\chi_{\fp_j}$ is completely determined by $t^{\chi}$.

  Now let $\fq$ be a prime not dividing $p$. 
  Then $\chi_{\fq}$ is determined by its value on $\xi_\fq$ (embedded into $K_\fq$ via $\sigma_{\fq}$) and
  vanishes on $\OO^\ast_\fq$.
  Hence 
$\chi_{\fq}$ is completely determined by $t^{\chi}$, since
\begin{equation*}
    \chi_\fq(\sigma_\fq(\xi_\fq)) = -\sum^m_{i=1}\chi_{\fp_i}(\sigma_i(\xi_\fq)) =
    -\sum^m_{i=1}t^{\chi}_{i}(\log_i(\sigma_i(\xi_\fq)))\,.\qedhere 
\end{equation*}
  \end{proof}
More generally, the proof of Lemma~\ref{L:chi_t} shows that every $t \in (K\otimes \Q_p)^\vee $ determines a character $\chi\colon
\mathbb{A}_K^\ast \to \Q_p$ via~\eqref{tdefnd}.  
We now investigate which $t \in (K\otimes \Q_p)^\vee $ give rise to a continuous idele class
character.
\begin{lemma}\label{L:icc_algo}
  A homomorphism
\[
  t = (t_{1},\ldots,t_{m}) \in (K\otimes\Q_p)^\vee
\]
gives rise to a continuous idele class character $\chi$ if and only if
\begin{equation}\label{E:unit_eqns}
  t_{1}(\log_1(\sigma_1(\eps)))+\ldots+ t_{m}(\log_m(\sigma_m(\eps)))=0
\end{equation}
  for all $\eps\in \OO_K^\ast$.
\end{lemma}
\begin{proof}
  Suppose that $\chi$ is a continuous idele class character, determined by $t=t^{\chi}\in
  (K\otimes\Q_p)^\vee$ as in Lemma~\ref{L:chi_t}.
  Note that for a unit $\eps\in \OO_K^\ast$ the value $\chi_{\fq}(\eps)$ depends only
  on $t_{\fq}$ for $\fq\mid p$ and is $0$ for all other primes $\fq$.
  Hence~\eqref{E:unit_eqns} must be satisfied.
  
  Conversely, suppose that $t \in (K\otimes\Q_p)^\vee$ and let
  $\chi$ denote the character $\chi \colon \A_K^\times\to \Q_p$  associated to $t$ as in Lemma~\ref{L:chi_t}.
Given $\beta\in K^*$, we need to show that $\chi(\beta)=0$. Suppose that
$(\beta) = \prod {\fq}_i^{e_i}$ is the decomposition of the fractional ideal
$(\beta)$, for primes $\fq_i$ of $K$. 
  Then
\begin{equation*}
  (\beta^{h_K}) = \prod \fq_i^{h_Ke_i} = \prod (\xi_{\fq_i})^{e_i} = \left(\prod
  \xi_{\fq_i}^{e_i}\right),
\end{equation*}
so that $\beta^{h_K} = \prod \xi_{\fq_i}^{e_i} u$ for some unit $u\in \OO_K^*$. But
since the construction of $\chi$ was done so that $\chi(\xi_\fq)=0$
for every $\fq$ and $\chi(u)=0$ for every unit, we see that
$\chi(\beta^{h_K})=0$ and hence $\chi(\beta)=0$.
\end{proof}
We deduce the following result:
\begin{corollary}\label{C:leo}
The continuous idele class characters of $K$ form a $\Q_p$-vector space
$V_K$ of dimension $\geq r_2 + 1$. 
If Leopoldt's conjecture~\cite{leopoldt1} holds for $K$, then we have $\dim_{\Q_p}V_K = r_2+1$.
In particular, this holds if $K/\Q$ is an abelian extension.
\end{corollary}
\begin{proof}
  The space $V_K$ has dimension precisely $d-r_K = r_2+1$
  if the
  equations~\eqref{E:unit_eqns} are independent; otherwise its dimension is larger.
  Leopoldt's conjecture \cite{leopoldt1} predicts that the $p$-adic regulator of $K$ is
  nonzero, which is equivalent to independence of the equations~\eqref{E:unit_eqns}.  
  Brumer~\cite{brumer} proved Leopoldt's conjecture for abelian $K/\Q$.
\end{proof}

\begin{remark}\label{R:mihailescu}
  For our applications it suffices that $\dim_{\Q_p} V_K \ge r_2+1$. 
\end{remark}

\begin{remark}\label{R:leo}
  Corollary~\ref{C:leo} is classical, see for instance~\cite{mazur:icm83}, but our proof yields a method
  for actually determining a basis of $V_K$, provided we know the class number $h_K$, fundamental units
and a method for computing, for each prime $\fq$ of $K$, a generator
$\xi_\fq\in K^\ast$ of the principal ideal $\fq^{h_K}$.
\end{remark}

\begin{example}\label{ex:cyc}
  If $t=(t_1,\ldots,t_m) \in (K\otimes\Q_p)^\vee$, where $t_j = \tr_{K_j/\Q_p}$, then
  the equation~\eqref{E:unit_eqns} translates into the trivially satisfied equation
  $\log(N_{K/\Q}(\eps)) = 0\,,$.
This choice of $t$ gives rise to a continuous idele class character $\lc$, which one
  obtains by composing the usual idele class
character over $\Q$ (corresponding to $t_1=\mathrm{id}_{\Q_p}$) with the norm on the idele class group.
We call $\lc$ the {\em cyclotomic} character; it cuts out the cyclotomic $\Z_p$-extension
of $K$.  
  Note that in the case of a totally real number field, if Leopoldt's conjecture holds, then every continuous idele class character is a scalar
multiple of $\lc$ by Corollary~\ref{C:leo}.
\end{example}

\begin{example}\label{ex:split}
If $p$ is totally split in $K$, then the trace maps correspond to scalar multiplication on
  $K_j\cong \Q_p$, so by Lemma~\ref{L:icc_algo}, the space $V_K$ of continuous idele class characters is isomorphic to the subspace of all $(c_1,\ldots,c_d) \in \Q_p^d$ such that
  \begin{equation}\label{E:unit_eqns_split}
  c_1\log(\sigma_1(\eps))+ \ldots+ c_d\log(\sigma_d(\eps)) = 0
  \end{equation}
  for all $\eps\in \OO_K^\ast$.
\end{example}

\begin{example}\label{ex:anti}
Let $K$ be imaginary quadratic.
Denote by $c\colon \A_K^\ast/K^\ast\to\A_K^\ast/K^\ast$ the involution induced by complex
conjugation.
We say that a continuous idele class character $\chi \in V_K$ is {\em anticyclotomic} if
$\chi\circ c = -\chi$. 
  Such characters cut out the anticyclotomic $\Z_p$-extension of $K$ and play an important
  role in Iwasawa theory.
The space $V_K$ of continuous idele class characters
is spanned by 
the cyclotomic character and any nontrivial anticyclotomic character.

If we assume, in addition, that $p$ splits completely in $K$, then it turns out that $\chi\in V_K$
is anticyclotomic if and only if \[c_1 +c_2 = 0\] in the notation of Example~\ref{ex:split}.
In this case, we call the character $\la$ corresponding to the choice $c_1 = 1$ and $c_2 =-1$
{\em the anticyclotomic} character. 
See~\cite[\S1]{WIN3} for an explicit construction of $\la$.
\end{example}

\subsection{Coleman-Gross $p$-adic height pairings}\label{S:padic_hts}
We keep the notation of \S \ref{S:icc}. 
Let $X/K$ be a smooth projective geometrically irreducible curve of genus $g>0$ with good reduction at
$\fp_1,\ldots,\fp_m$ and let $J$ denote its Jacobian. The rank $\rk(J/K)$ will be denoted
by $r$.
We fix, for every $j\in \{1,\ldots,m\}$, a splitting
\[
  \hdr^1(X/K_j) = H^0(X/K_j,\Omega^1)\oplus W_j\, ,
\]
such that $W_j$ is isotropic with respect to the cup product pairing.
For instance, when
$\fp_j$ is ordinary, we can take the unit root subspace with respect to Frobenius.

There are several ways of attaching a $p$-adic height pairing 
\begin{equation}\label{phtjac}
  h^\chi\colon J(K)\times J(K) \to \Q_p
\end{equation}
to a continuous idele class character $\chi \in V_K$,
see for instance the work of Schneider~\cite{Schn82},
Mazur-Tate~\cite{mazur-tate}, Coleman-Gross~\cite{Col-Gro89} and
Nekov{\'a}{\v{r}}~\cite{Nek93}. The pairing $h^{\chi}$ depends on our fixed choices of $W_j$.
If the reduction is ordinary, $W_j$ is the unit root subspace of Frobenius at all $j$, and  $\chi$ is ramified above $p$, the constructions of Schneider, Mazur-Tate, Coleman-Gross and Nekov{\'a}{\v{r}} are all known to be
equivalent.
The $p$-adic height pairing is bilinear and functorial, see~\cite[$\S3.4$]{mazur-tate}. 
It is furthermore symmetric, since we assume all $W_j$ to be isotropic.
Denote by $\Bil$ the space of $\Q_p$-valued
bilinear forms on $J(K)\otimes_{\Z} \Q_p$. The function $V_K \to
\Bil$ mapping $\chi$ to $h^{\chi}$ is $\Q_p$-linear.
The dimension of its image is bounded from above by $$\min\left\{\dim_{\Q_p}V_K,
\frac{r(r+1)}{2}\right\}\ge\min\left\{r_2+1,\frac{r(r+1)}{2}\right\}\,,$$
where the inequality is an equality if Leopoldt's conjecture holds for $K$
(see Corollary~\ref{C:leo}).
Lower bounds on the dimension are much more difficult. For instance, when the reduction is
ordinary, the cyclotomic height attached to the unit root splitting has been conjectured
to be non-degenerate by Schneider, but this remains open to date. 

As in~\cite{BBM14}, we follow the approach of Coleman-Gross.
Let $\chi = \sum_v\chi_v \in V_K$
be a nontrivial continuous idele class character, with corresponding trace maps
$t^\chi = (t^{\chi}_1,\ldots,t^{\chi}_m) \in (K\otimes\Q_p)^\vee$.
Coleman and Gross construct local bi-additive symmetric $p$-adic height pairings $h^{\chi}_v$ for every non-archimedean
place $v$ of $K$ on divisors of degree~0 with disjoint support. Then they define the global
$p$-adic height pairing to be their sum $h^\chi = \sum_vh^\chi_v$; this turns out to
respect linear equivalence and hence gives a well-defined symmetric bilinear
pairing~\eqref{phtjac}.

Following Gross~\cite[\S5]{Gross86}, the first- and second-named author
removed in \cite{BB:Crelle} the condition that the divisors have to be relatively prime for the local
$p$-adic height pairings to make sense.
For this, $p$-adic Arakelov theory as developed by the second-named author~\cite{Bes00} was used.
This extension to arbitrary divisors of degree~0 requires the choice of a
tangent vector for every $Q \in X$. If one chooses such vectors consistently for all
primes, then one obtains that the sum of the local height pairings is independent of this
choice.

We now recall the construction of the local height pairings.
First let $\fq$ be a prime of $K$ not dividing $p$. 
Then the local height pairing $h^\chi_{\fq}$ determined by $\fq$ is defined in terms of
intersection theory, as follows: Let $D, D'$ be divisors on $X/K_{\fq}$ of degree 0 and with
disjoint support.
Fix a proper regular model $\XX$ of $X/ K_\fq$ over $\Spec \OO_{K_\fq}$ and extend $D$ and
$D'$ to $\Q$-divisors $\DD,\DD'$ on $\XX$ which have intersection multiplicity~0 with every
vertical divisor. Then we define, as in~\cite{Col-Gro89},
\begin{equation}\label{away_disjoint}
  h^\chi_{\fq}(D, D') \colonequals \chi_{\fq}(\pi_{\fq})(\DD\cdot\DD')\,,
\end{equation}
where  $\pi_\fq$ is a uniformizer in $K_\fq$ and $(\DD\cdot\DD')\in \Q$ is the intersection multiplicity. 
A choice of tangent vector makes it possible to define the intersection between two
divisors with common component, which can be used to extend the
definition~\eqref{away_disjoint} in a straightforward way, see~\cite[\S6]{Gross86} and \cite[Proposition~2.4]{BB:Crelle}.
One can define the height pairing at all $\fp_j$ such that $\chi_j$ is unramified
in an analogous way.

Let $j \in \{1,\ldots,m\}$ such that $\chi_j$ is ramified and consider the trace map
$t^{\chi}_j \colon K_j \to \Q_p$.
Suppose that $D$ and $D'$ are divisors of degree~0 on $X/ K_j$. 
If they have disjoint support, then their local height pairing at $\fp_j$ determined by
$\chi$ is defined as
\[
  h^{\chi}_{\fp_j}(D,D') \colonequals t^{\chi}_j \left(\int_{D}\omega_{D'}\right)\, ,
\]
where $\omega_{D'}$ is a differential of the third kind with residue divisor $D'$, normalized
with respect to $W_j$ as in~\cite[\S3]{Col-Gro89}. The integral is a Coleman integral; it
depends on the branch $\log_j\colon K_j^\ast\to K_j$ of the logarithm determined by
$\chi$, see Remark~\ref{ramified}.
To extend this to divisors with common support, one uses that 
$$
  h^{\chi}_{\fp_j}(D,D') = t^{\chi}_j(G_{D}(D')) 
$$
where $G_{D}$ is the $p$-adic Green function on $X_{ K_j} (\overline{K_j})\setminus
\supp(D)$ defined in~\cite{Bes00}.
The choice of a tangent vector at every point of $X$ makes it possible to define
$G_{D}[D']\in K_j$
for any $D'$ of degree~0, so that $G_{D}[D'] = G_{D}(D')$ when $D$ and $D'$ have
disjoint support.
The general formula for the local height pairing of $D$ and $D'$ given
in~\cite[Proposition~3.4]{BB:Crelle} is then
$$
  h^{\chi}_{\fp_j}(D,D') = t^{\chi}_j(G_{D}[D']).
$$
\begin{example}\label{ex:ht_cyc}
The local cyclotomic height pairings 
for $D,\,D' \in \Div^0(X)$ with disjoint support
 are given by 
\[  
  \hc_{\fp_j}(D, D') = \tr_{K_j/\Q_p}\left(\int_{\sigma_j(D')}\omega_{\sigma_j(D)}\right)
\]
  for $j\in\{1,\ldots,m\}$ 
and 
\[
  \hc_{\fq}(D, D') = -\frac{1}{h_K}\log(N(\xi_\fq))\cdot (D,D')_{\fq}\,,
\]
for primes $\fq\nmid p$.
\end{example}

\begin{example}\label{ex:ht_anti}
  We now describe the anticyclotomic height $\ha \colonequals h^{\la}$ when $K$ is imaginary
  quadratic, see Example~\ref{ex:anti}.
In our notation, the local heights $\ha_{\fp_i}$ are 
\[  
    \ha_{\fp_1}(D, D') = \int_{\sigma_1(D')}\omega_{\sigma_1(D)}\;\textrm{ and }\;
    \ha_{\fp_2}(D, D') = -\int_{\sigma_2(D')}\omega_{\sigma_2(D)}\,.
\]
and the local heights away from $p$ are given by
\[
  \ha_{\fq}(D, D') = \frac{1}{h_K}\log\left(\frac{\sigma_2(\xi_\fq)}{\sigma_1(\xi_\fq)}\right)\cdot
  (D,D')_{\fq}\,.
\]

Using these formulas, we get a practical method for computing $\ha$. See \cite{WIN3} for
  an alternative algorithm.
\end{example}

\section{Chabauty over number fields}\label{S:cnf}
In this section, we recall a refinement of the method of Chabauty and Coleman over number fields.
The idea is that by viewing the $K$-rational points on $X/K$ as the
$\Q$-rational points on the Weil restriction of scalars $\Res_{K/\Q}(X)$ and using
$\Q_p$-valued (rather than $K_{\mathfrak{p}}$-valued, for a single $\mathfrak{p}\mid p$) integrals, one can often compute
$X(K)$ when the rank of the Jacobian over $K$ is greater than $g$.
This was first suggested by Wetherell at a talk at MSRI in~2000, but he never published the details. Siksek ~\cite{siksek:ecnf} later gave a Chabauty-Coleman method over number fields inspired by Wetherell's work. This allows one to compute $X(K)$ 
in many cases using a combination of his
idea with the Mordell-Weil sieve. See also Triantafillou's
work~\cite{Nick19} for a recent application to hyperbolic curves of genus 0 and $S$-unit
equations.

We keep the notation of the previous section. We also assume that there is a rational
point $P_0\in X(K)$, and we let $\iota\colon X \hookrightarrow J\colonequals \Jac(X)$ denote the Abel-Jacobi map corresponding to it.
We write  $ X(K\otimes\Q_p)$  for the subset $X(K_1)\times\cdots\times X(K_m)$ of
$ X(\Qpb)^m$. Let $\sigma \colon X(K) \hookrightarrow X(K\otimes\Q_p)$ be the embedding
  induced by $\sigma_j  \colon  X(K) \hookrightarrow X(K_j)$.

Siksek's goal is to explicitly
construct $d=[K:\Q]$ locally analytic functions 
\[
  \rho_1,\ldots,\rho_d \colon X(K\otimes\Q_p)\to \Q_p
\]
such that
\begin{enumerate}[label=(\roman*)]
  \item we have $\rho_l(\sigma(X(K)))=0$ for all $1 \le l \le d$;
  \item the set
    \[
      B \colonequals \left\{\zvec \in X(K\otimes\Q_p) \, \colon \, \rho_l(\zvec) =0\;\;\textrm{for
    all}\;\;1\le l \le d\right\}
  \]
  is finite.
\end{enumerate}
We need (at least) $d$ such functions to cut out a finite subset of $X(K\otimes \Q_p)$ for
reasons of dimension.

With a view toward explicit computations and generalizations, we now discuss how 
the functions $\rho_l$ can be constructed in terms of Coleman integrals.
Let $j \in \{1,\ldots,m\}$ and let $z_j \in X(K_j)$.
If $\omega$ is a holomorphic~1-form on $X/K_{j}$, then the integral
$\int_{P_0}^{z_j}
\omega$ is an element of $K_j$. We can compose it with a trace map
$t_j \in \Hom_{\Q_p}(K_j,\Q_p)$ to obtain a locally analytic function
\[
  X(K_j)\to \Q_p\,;\;\;  z_j \mapsto t_j\left(\int^{z_j}_{P_0} \omega\right)\,.
\]
For every $j \in\{1,\ldots,m\}$, we fix a basis $(t_{j,1},\ldots,t_{j,d_j})$ of
$\Hom_{\Q_p}(K_j,\Q_p)$, where $d_j = [K_j:\Q_p]$; we also fix a basis
$(\omega_0,\ldots,\omega_{g-1})$ of $H^0(X/K,\Omega^1)$.
\begin{definition}\label{D:fijk}
  Let $i \in \{0,\ldots,g-1\}$, let $j \in\{1,\ldots,m\}$ and let $k \in
  \{1,\ldots,d_j\}$.
We define 
    \[
      f^{(j,k)}_i(\zvec) \colonequals
      t_{j,k}\left(\int^{z_j}_{P_0}\sigma_j(\omega_i)\right)\in\Q_p
  \]
  for $\zvec= (z_1,\ldots,z_m) \in X(K\otimes\Q_p)$.
For $Q \in X(K)$, we set
  \[
    f^{(j,k)}_i(Q) \colonequals f^{(j,k)}_i(\sigma(Q))\,.
  \]
\end{definition}
We fix an ordering of the $f^{(j,k)}_i$ and label their restrictions to $X(K)$ as follows:
\[
  f_0,\ldots,f_{dg-1} \colon  X(K) \to \Q_p\,.
\]
Extending these functions to Galois-equivariant functions $f_i\colon X(K\otimes
\bar{\Q}_p)\to \bar{\Q}_p$ in the obvious way and then to $J(K)\otimes_{\Z}\Q_p$   by additivity, we obtain $\Q_p$-valued linear functionals
\begin{equation}\label{E:fiJ}
  f_0,\ldots,f_{dg-1} \colon J(K)\otimes_{\Z}\Q_p\to \Q_p\,.
\end{equation}
We fix some more notation:
Let $(J(K)\otimes\Q_p)^\vee = \Hom_{\Q_p}(J(K)\otimes_{\Z}\Q_p, \Q_p)$, let
\[
  U = \mathrm{Span}_{\Q_p}(f_0,\ldots,f_{dg-1})\subset (J(K)\otimes\Q_p)^\vee\,,
\]
and let 
\[
  n = dg - \dim_{\Q_p}(U)
\]
denote the maximal number of independent relations between the $f_i$.

Suppose that 
\begin{equation}\label{ndcond}
  n \ge d
\end{equation}
and fix a set of $d$ independent relations in $U$.
These relations, pulled back via the Abel-Jacobi map $\iota$, extend to explicitly computable locally analytic functions
\[
  \rho_1,\ldots,\rho_d  \colon X(K\otimes\Q_p) \to \Q_p
\]
such that
\[
\rho_l(\sigma(X(K)))=0
\]
for all $l=1,\ldots,d$.
Then, provided the common zero set $B$ of $\rho_1,\ldots,\rho_d$ is finite,
we can explicitly determine $X(K)$, by first computing $B$ (to suitable $p$-adic
precision) and then identifying $\sigma(X(K)) \subset B$ using the Mordell-Weil sieve,
if $g\ge 2$.
See~\cite{siksek:ecnf} for details.

As mentioned previously, over $\Q$ the condition $r<g$ is a sufficient criterion for the
method of Chabauty-Coleman to work.
But when $K\ne\Q$, it is much more delicate to predict whether the zero set $B$ is finite, as
we have to deal with systems of multivariate power series 
(see also Appendix \ref{appendix:roots_multi}).
Since  $\dim_{\Q_p}(J(K)\otimes\Q_p)^\vee = r$,
we have
$ n \ge dg-r$. Therefore one might guess that for finiteness of $B$ it should be sufficient that
\begin{equation}\label{E:samir_cond}
  r \le d(g-1).
\end{equation}
This is not the case, as one can find examples where $B$ is infinite because $X$ can be defined over some subfield $F$ of $K$ over which $\rk(J/F)>[F\colon
\Q](g-1)$; see~\cite{siksek:ecnf} for a discussion.
Siksek asks whether $B$ is finite whenever 
\begin{equation}\label{E:cond_sub}
  \rk (\Jac(Y)/F) \le [F:\Q](g-1)
\end{equation}
for all subfields $F$ and for every smooth projective curve $Y/F$ such that
$Y\times_F K
\cong_K X$.
This was shown to be false by Dogra~\cite{Dog19}; he constructs a hyperelliptic genus 3 curve
$X$ with minimal field of definition $K\colonequals \Q(\sqrt{33})$
covering a genus 2 curve $X_0/\Q$ 
whose Jacobian $J_0$ satisfies $\rk(J_0/\Q)=2$. This $X$ satisfies the
condition~\eqref{E:cond_sub}, but the functions $\rho_l$ vanish in the preimage
of $X_0(\Q_p)$ in $X(K\otimes \Q_p)$, so $B$ is infinite.
Dogra also shows that $B$ is finite whenever $r\le d(g-1)$ and $\Hom(J(K)_{\bar{\Q},\sigma_1}, J(K)_{\bar{\Q},\sigma_2})$ is trivial for any two
distinct embeddings $\sigma_1,\sigma_2\colon K\hookrightarrow\bar{\Q}$, i.e. whenever $J$ and its
conjugates do not share an isogeny component over $\bar{\Q}$.

When $r>d(g-1)$, then the method of Siksek is in general not applicable. In the following
two sections, we discuss how to extend it using $p$-adic heights.

\section{Quadratic Chabauty for integral points on odd degree hyperelliptic curves over number fields}\label{S:qcnf}
Suppose that $X/K$ is hyperelliptic, given by an equation $X\colon y^2=f(x)$, where
$f \in \OO_K[x]$ is a separable polynomial of degree $2g+1\ge 3$ which
does not reduce to a square modulo any prime. 
Let $\iota$ be the Abel-Jacobi map with respect to $\infty=(1:0:0) \in X(K)$. 
We denote by $\UU$ the affine scheme $\Spec \OO_K[x, y] / (y^2 - f(x))$.

For the case $K=\Q$, the first, second and fourth author showed in~\cite{BBM14, BBM16} how to compute $\UU(\Z)\subset X(\Q)$
when $r=g$ using a locally analytic function constructed in terms of the cyclotomic $p$-adic height. 
We now show how to generalize this idea to arbitrary number fields. Generically, we expect
our method to work when 
\begin{equation}\label{ranks}
r+r_K\le dg\,.
\end{equation}

More precisely, we try to 
construct explicitly computable functions 
\[\rho_1,\ldots,\rho_d \colon  \UU(\OO_K\otimes\Z_p) \to \Q_p
\]
which are locally analytic on $p$-integral points 
and explicitly computable finite subsets $T^{(1)}, \ldots, T^{(d)} \subset \Q_p$
such that
\begin{enumerate}[label=(\roman*)]
  \item we have $\rho_l(\sigma(\UU(\OO_K)))\subset T^{(l)}$ for all $1 \le l \le d$;
  \item the set 
    \[
      B \colonequals \left\{ \zvec \in \UU(\OO_K\otimes\Z_p) \, \colon \, \rho_l(\zvec) \in T^{(l)}\;\;\textrm{for
    all}\;\;1\le l \le d\right\}
  \]
  is finite. 
\end{enumerate}
As in Section \ref{S:cnf}, let $U$ be the $\Q_p$-span of the linear functionals $f_1,\dots, f_{dg-1}\colon J(K)\otimes \Q_p\to \Q_p$.  We show how to construct such functions $\rho_1,\ldots,\rho_d$ when~\eqref{ranks} and the following condition are
satisfied:
\begin{condition}\label{C:crucial}
  $U = (J(K)\otimes\Q_p)^\vee$.
\end{condition}
The finiteness of $B$ turns out to be much more difficult to prove.
If Condition~\ref{C:crucial} is satisfied, then we have $\dim_{\Q_p} U = r$ and hence 
\begin{equation}\label{rank_cond}
  n+r = dg\,;
\end{equation}
in particular, we need to have $r \le dg$ for Condition~\ref{C:crucial} to hold,
and if $r=dg$, then Condition~\ref{C:crucial} is equivalent to surjectivity of  the map in~\eqref{log2}.
We order the $f_i$ so that $(f_0,\ldots,f_{r-1})$ is a basis of
$(J(K)\otimes\Q_p)^\vee$.

In \S\ref{S:padic_hts} we discussed how a continuous idele class character $\chi$
gives rise to a $p$-adic height pairing $h^\chi$.
So let $\chi \in V_K$ be nontrivial with associated trace maps $t^\chi =
(t^{\chi}_1,\ldots,t^{\chi}_m) \in (K\otimes\Q_p)^\vee$.
We choose the basis  $(\omega_0,\ldots,\omega_{g-1})$ 
of $H^0(X/K,\Omega^1)$ given by $\omega_k \colonequals x^k dx/2y$
and we denote the differentials on $X/K_j$ obtained via $\sigma_j$ by $\omega_k$ as well. 
In order to define the local height pairings  for arbitrary divisors 
of degree~0, we fix a choice of tangent vector for every point $Q \in X$ as in~\cite{BBM14}.
Namely, we take the tangent vector induced by $\omega_0$ by duality for all affine points. 
For the point $\infty$, we pick the dual of the value of $\omega_{g-1}$ at
$\infty$.
We denote by $(\bom_0,\ldots, \bom_{g-1})$ the
unique basis of $W_j$ which is dual to $(\omega_0,\ldots,\omega_{g-1})$ with respect to the cup product pairing.
For $z \in X({K_j})\setminus\{\infty\}$ we define the double Coleman integral 
$$\tau_j(z) = -2\int_{b_0}^z\sum_{i=0}^{g-1}\, \omega_i\bom_i\in K_j\,,$$
where $b_0$ is the tangential basepoint determined by our choice of tangent vector at
$\infty$.
Note that $\tau_j$ depends on the branch $\log_j\colon K_j^\ast\to K_j$ determined by
$\chi$. This is not reflected in our notation, since we are mostly interested
in the restriction of $\tau_j$ to $\UU(\OO_j)$.
\begin{theorem}\label{T:rhos}
  Suppose that Condition~\ref{C:crucial} is satisfied.
  Then there are explicitly
  computable constants $\alpha^\chi_{ij} \in\Q_p$ and an explicitly computable finite set $T^{\chi}\subset \Q_p$  such that the function 
  \[
    \rho^\chi \colon \UU(\OO_K\otimes\Z_p) \to \Q_p
\]
defined by
\[
  \rho^\chi(\zvec) \colonequals \sum^m_{j=1}t^{\chi}_j(\tau_j(z_j)) -
  \sum_{0 \le i\le j\le r-1}\alpha_{ij}^\chi f_i(\zvec)f_j(\zvec)\,,
\]
satisfies
\[
  \rho^\chi(\sigma(\UU(\OO_K))) \subset T^{\chi}\,.
\]
\end{theorem}
\begin{proof}
Setting $g_{ij}(P,Q)\colonequals\frac{1}{2}(f_i(P)f_j(Q) + f_j(P)f_i(Q))$, we obtain a basis
$(g_{ij}  \colon  0\le i\le j \le r-1)$
  of the space of $\Q_p$-valued bilinear forms on $J(K)\otimes_{\Z} \Q_p$, since the latter has dimension
$r(r+1)/2$ and the $g_{ij}$ are independent.
Because $h^\chi$ is a $\Q_p$-valued bilinear form on $J(K)\otimes_{\Z} \Q_p$, we can
find constants $\alpha^\chi_{ij} \in\Q_p$ such that
\begin{equation}\label{h_gij}
  h^\chi = \sum_{0\le i\le j \le r-1} \alpha^{\chi}_{ij}g_{ij}\,.
\end{equation}
  If $\fq$ is a prime of $K$ and $z \in X(K_\fq)\setminus\{\infty\}$, 
  we write
\[
  h_\fq^\chi(z)  \colonequals  h_\fq^\chi((z)-(\infty), (z)-(\infty))\,,
\]
where the right hand side is determined by our choice of tangent vectors.
Then, we have
\begin{equation}\label{height-tau}
  h^\chi_{\fp_j}(z)=t^{\chi}_j(\tau_j(z))
\end{equation}
  for $z \in X(K_j)\setminus\{\infty\}$ 
  by~\cite[Theorem~2.2]{BBM14}.

Now let $\fq$ be a prime of $K$ not dividing $p$. 
Let $\XX$ be a desingularization in the strong sense of
the Zariski closure of $X$ in weighted projective space $\PP_{\OO_{K_\fq}}(1,g+1,1)$.
Let $z \in X(K_\fq)$ and extend the divisor $(z)-(\infty)$ to a $\Q$-divisor 
$\DD_z$ on $\XX$
such that $\DD_z$ has intersection multiplicity~0 with all vertical divisors on $\XX$.
  Then we have $h^{\chi}_\fq(z)= \chi_\fq(\pi_\fq)\DD_z^2\,.$ 
The proof of~\cite[Proposition~3.3]{BBM14} (which treats the special case $K=\Q$ and
$h=\hc$) shows that 
if $z \in \UU(\OO_{K_\fq})$, then the value of $h^\chi_{\fq}(z)$ depends only on the
irreducible component $\Gamma_z$ of the special fiber of $\XX$ that $z$ reduces to and is
  explicitly computable from $\Gamma_z$. 
Furthermore, this value is $0$ if $z$ reduces to a smooth point modulo $\fq$ and if $\ord_\fq(\atp) =
0$, where $\atp$ is the leading coefficient of $f$. 

For $Q \in \UU(\OO_{K})$ we conclude that
\begin{align*}
  \rho^\chi(\sigma(Q)) &= \sum^m_{j=1}t^{\chi}_j(\tau_j(\sigma_j(Q))) - \sum_{0 \le i\le j\le
  r-1}\alpha_{ij}^\chi f_i(Q)f_j(Q) \\
  &= \sum^m_{j=1}h^\chi_{\fp_j}(Q) - h^{\chi}((Q)-(\infty), (Q)-(\infty))\\
  &= -\sum_{\fq \nmid p} h^\chi_{\fq}(Q) 
\end{align*}
indeed takes values in an explicitly computable finite set.
\end{proof}
\begin{remark}
  Theorem~\ref{T:int_intro} follows at once from Theorem~\ref{T:rhos}.
\end{remark}

\begin{remark}
Let us make the following remark concerning the dependence on the branch of the logarithm.
  The function $\tau_j$ is a double Coleman integral. On residue discs where the
  integrands have no singularities, it is rigid analytic. However, there will be a finite
  number of residue discs where $\tau_j$ will be more complicated. Near a point where one
  of the integrands has a singularity, there will be a disc $D$ where the function is given
  by a polynomial of degree at most $2$ in $\log(z)$ with coefficients which are rigid
  analytic functions on $D$, with $z$ a uniformizing parameter on $D$ sending the singular
  point to $0$ (these discs could in general be smaller than residue discs if the
  integrands have several singular points in the same residue disc). We note that in fact,
  due to their source in Green functions, the $\tau_j$ are simpler: On such a disc $D$
  they will be of the form $\varphi(z)+ c\log(z)$, where $\varphi$ is rigid analytic and $c$ is a constant. Solving even a simple polynomial equation involving both $z$ and $\log(z)$ is far more complicated, if possible at all, than an analytic equation. This is one reason why certain quadratic Chabauty techniques do not find all rational points but only those that avoid certain residue disks. On the other hand, whenever one stays away from these problematic disks the equation does not depend on the chosen branch of the logarithm. The dependence on the branch only appears when trying to compute the global height pairing.    
\end{remark}

\begin{remark}\label{R:}
  Our assumption that $f$ does not reduce to a square modulo any prime is only
  required to apply~\cite[Proposition~3.3]{BBM14}. Of course, we can always scale the
  variables to make $f$ monic, and then it is automatically satisfied.
\end{remark}

We can expect at most
$\min(\dim_{\mathbb{Q}_p}V_K,\frac{r(r+1)}{2})$ independent height pairings $h^{\chi}$. 
Recall that $\dim_{\Q_p} V_K \ge r_2+1$ (with equality if Leopoldt's conjecture holds for
$K$). 
The
set of all $\zvec\in \UU(\OO_K\otimes\Q_p)$ such that $\rho^\chi(\zvec)\in T^\chi$ holds for the corresponding functions
can only be finite when $r_2+1=d$, i.e. when $K=\Q$ or $K$ is an imaginary
quadratic  field.

In general, we need  at least $d-(r_2+1) = r_K$ additional functions $\UU(\OO_K\otimes\Q_p) \to \Q_p$ to cut out a finite subset of $\UU(\OO_K\otimes\Q_p)$ containing $\sigma(\UU(\OO_K))$.
It is natural to take these from the relations among the functions $f_i$, as
in~\S\ref{S:cnf}.
Therefore we need the maximal number $n$ of independent relations to be at least $r_K$.
Because of~\eqref{rank_cond}, this means that we want $r_K+r \le dg$, which is
precisely~\eqref{ranks}, to hold.

Let $\rho_1,\ldots,\rho_{r_2+1} \colon \UU(\OO_K\otimes\Z_p) \to \Q_p$ be functions coming from $p$-adic
heights associated to nontrivial continuous idele
class characters $\chi \in V_K$ as above, with corresponding finite sets $T^{(k)} \subset
\Q_p$. Assuming that~\eqref{ranks} holds, we obtain functions  $\rho_{r_2+2},\ldots,\rho_d
\colon X(K\otimes\Q_p) \to \Q_p$ which vanish in $\sigma(X(K))$
from independent relations among the $f_i$ as in Section~\ref{S:cnf}. Set $T^{(l)} = \{0\}\subset \Q_p$ for
$r_2+2\le l \le d$.
Then
\[
  B \colonequals \{\zvec \in \UU(\OO_K\otimes\Z_p) \, \colon \, \rho_l(\zvec) \in T^{(l)}\;\;\textrm{for all}\;\;1\le l \le d\}
\]
contains $\sigma(\UU(\OO_K))$ and is explicitly computable.
Thus, we get 
a method to $p$-adically approximate $\UU(\OO_K)$ in practice whenever $B$ is
finite.

In analogy with the results of Dogra~\cite{Dog19}, we raise the following
\begin{question}\label{rk_question}
  Suppose that Condition~\ref{C:crucial} is satisfied and that we have
\begin{equation}\label{E:qc_cond}
  \rk (\Jac(Y)/F) +r_F \le [F:\Q]\cdot g
\end{equation}
for every subfield $F\subset K$ and for every smooth projective curve $Y/F$ such that $Y\times_F K
  \cong_K X$. Is it true that $B$ can only be infinite for geometric reasons?
\end{question}
\begin{remark}
  Siksek's method always succeeds in constructing a set $B$ containing $X(K)$, provided that $r\leq d(g-1)$. As was explained at the end of Section \ref{S:cnf}, this assumption on the rank is however not sufficient to guarantee finiteness of $B$. On the contrary, even definability of a set $B$ is not ensured by \eqref{E:qc_cond}. If $K=\Q$ and \eqref{E:qc_cond} holds, then either Condition~\ref{C:crucial} is satisfied or classical Chabauty is applicable (or both). For general $K$, it is instead possible to construct examples where \eqref{E:qc_cond} holds, but we are neither in the situation of Condition~\ref{C:crucial} nor in Siksek's. Indeed, let $X$ be a genus $2$ hyperelliptic curve over an imaginary quadratic field $K$. Suppose that $X$ does not not descend to $\Q$ and that its Jacobian splits as the product of two elliptic curves $E_1$ and $E_2$. If $\rk(E_1/K)=3$ and $\rk(E_2/K)=1$ and $E_1$ cannot be defined over $\Q$, then we expect that $\dim_{\Q_p}(U) = 3$. For example, consider the hyperelliptic curve over $\Q(i)$
\begin{equation*}
X\colon y^2 = (14i - 6)x^5 + (-3i + 5)x^4 + 20ix^3 - 15x^2 - 6ix + 1.
\end{equation*}
Its Jacobian splits as the product of
\begin{align*}
E_1&\colon y^2 = x^3 + (-3i - 10)x^2 + (54i + 72)\\
E_2 &\colon  y^2 = x^3 + (-3i - 10)x + (-3i - 9)
\end{align*}
of respective ranks $3$ and $1$. Let $P=(2, -3i - 7),\,Q= (7i + 1,-6i + 22)\in E_1(\Q(i))$ and let $\omega_0$ be an invariant differential on $E_1$. Then, choosing $p=13$, we get $(\int_{\infty}^{\sigma_1(P)}\omega_0)(\int_{\infty}^{\sigma_2(Q)}\omega_0)\neq (\int_{\infty}^{\sigma_2(P)}\omega_0)(\int_{\infty}^{\sigma_1(Q)}\omega_0)$, so \eqref{log2} is surjective for $E_1$. Thus the dimension of $U$ for $J$ is $3$.

When $K=\Q$ it is conjectured by Waldschmidt that it is sufficient that $J$ is simple of
  Mordell-Weil rank at most 
$g$ for Condition~\ref{C:crucial} to hold (see \cite{waldschmidt}).
  The situation is expected to be more complicated for higher degree number fields,
  see~\cite[Remark~6.4]{Poonen:Leopoldt}.
\end{remark}

\begin{remark}
The equations attached by Theorem \ref{T:rhos} to dependent height pairings can sometimes be independent.  To see this, consider an elliptic curve $E$ over $\Q$ of rank $0$. Then all global height pairings are identically zero on $E(\Q)$; nevertheless, the cyclotomic height (with respect to a choice of splitting) gives an equation whose zero set is, in general, not identical to $E(\Q_p)_{\mathrm{tors}}$, i.e. to the zero set of the $p$-adic logarithm on $E(\Q_p)$. See \cite{BCKW12,bia:QCKim} for computational evidence and relations with a conjecture of Kim on effectiveness.

This phenomenon in rank $0$ also provides us with an example of how different choices of splittings of the de Rham cohomology can lead to linearly independent equations. Indeed, recall from Section \ref{S:padic_hts} that the cyclotomic $p$-adic height pairing on $E/\Q$ depends on a choice of subspace $W\subset \hdr^1(E/\Q_p)$, complementary to the space of holomorphic form and isotropic with respect to the cup product pairing. The latter condition is automatically satisfied in the elliptic curve case and we see that the local height at $p$ is then given by
\begin{equation*}
h_p^{(W)}(z) = 2D_2(z)+c_W\left(\int_{\infty}^{z}\omega_0\right)^2\,,
\end{equation*}
where $D_2(z)$ is the dilogarithm (a double Coleman integral), $c_W\in \Q_p$ depends on $W$, and $\int_{\infty}^{z}\omega_0$ is the $p$-adic logarithm. Choosing two different subspaces $W_1,W_2$ we then get $h_p^{(W_1)}(z)=0=h_p^{(W_2)}(z)$ if and only if $h_p^{(W)}(z)=0=\int_{\infty}^{z}\omega_0$ for any $W$. Thus, for all $z\in \mathcal{U}(\Z_p)$ such that $h_p^{(W_1)}(z)=0$ but $\int_{\infty}^{z}\omega_0\neq 0$, we have $h_p^{(W_2)}(z)\neq 0$; examples of such $z$ can be found in the computations for \cite{BCKW12,bia:QCKim}.

In contrast, when $r=g$, different splittings will not result in independent equations, see~\cite[Remark 3.12]{BDMTV}.
\end{remark}

\section{Quadratic Chabauty for rational points on bielliptic curves over number
fields}\label{sec:rat}
Let $X/K$ be the genus $2$ bielliptic curve
\begin{equation*}
y^2 = a_6x^6 + a_4x^4+a_2x^2+a_0\qquad (a_i\in \OO_K)
\end{equation*}
which has degree $2$ maps $\varphi_1$, $\varphi_2$ to two elliptic curves:
\begin{align}
\label{eq:ell_of_biell}
E_1&\colon y^2=x^3+a_4x^2+a_2a_6x+a_0a_6^2,\qquad &&\varphi_1(x,y) = (a_6x^2,a_6y)\\
E_2&\colon y^2=x^3+a_2x^2+a_4a_0x+a_6a_0^2, &&\varphi_2(x,y) = (a_0x^{-2},a_0yx^{-3}). \nonumber
\end{align}
When $K$ is either $\mathbb{Q}$ or an imaginary quadratic field and both $E_1$ and $E_2$
have Mordell-Weil rank $1$ over $K$, it was shown in \cite{BD18} that a suitable choice of continuous idele class character of $K$ gives rise to a locally analytic function on $X(K_{j})$ which vanishes on $X(K)$, provided that $K_{j}\simeq \Q_p$ and $X$ has good reduction at each $\fp_j$.

In this section, we combine the ideas of \cite{BD18} with those of Section \ref{S:qcnf} to approximate $X(K)$ inside $X(K\otimes \mathbb{Q}_p)$ when $K$ is an arbitrary number field and $p$ is a prime, unramified in $K$, above which the given model for $X$ has everywhere good reduction. As in Sections \ref{S:cnf} and \ref{S:qcnf}, we do not investigate here whether the systems of $d$ equations in $d$ variables that we define have finitely many zeros.

For ease of exposition, and since we assumed in Section \ref{S:cnf} that $X$ possesses a $K$-rational point, we now restrict to curves $X$ for which the defining polynomial is monic, i.e.\ for the rest of this section we take
\[a_6 = 1.\]
 Then $X$ can be embedded into its Jacobian $J$ via the Abel-Jacobi map with respect to one of $\infty^{\pm}=(1: \pm 1: 0)\in X(K)$.
 The proof of the main theorem (Theorem \ref{thm:bielliptic} below) can easily be adapted to allow for any $a_6\in \OO_K$.

For the construction, it is sufficient to work under Assumption \ref{ass:E1E2} below and
\begin{equation}
\label{eq:rkgenusbielliptic}
\rk(J/K)+ r_K \leq dg =2d.
\end{equation} 
 
For each $A\in\{J, E_1, E_2\}$, let 
\[
  f^{A}_0,\ldots,f^{A}_{dg(A)-1} \colon  A(K)\otimes_{\Z}\Q_p \to \Q_p\,
\]
be the $\Qp$-valued linear functionals of Section \ref{S:cnf} on $A$ (here by $g(J)$ we mean the genus of $X$) and define, for $0\leq i\leq j\leq dg(A)-1$, the bilinear form $g_{ij}^A$ by
\[
g_{ij}^{A}(P,Q) = \frac{1}{2}(f_i^A(P)f_j^A(Q)+f_j^A(Q)f_i^A(P)),\quad \text{for } P,Q\in A(K)\otimes \Q_p.
\]
Let $\chi$ be a nontrivial continuous idele class character of $K$ and for each $k\in\{1,2\}$ let $h^{\chi, E_k}$ be the global $p$-adic height pairing of Section \ref{S:hts} on $E_k$ with respect to $\chi$. Similarly, for a prime $\fq$ of $K$, denote by $h_{\fq}^{\chi, E_k}$ the local height at $\fq$. We work under the following assumption:
\begin{assumption}
\label{ass:E1E2}
Condition \ref{C:crucial} holds for each of $E_1$ and $E_2$. For this, it is necessary that 
\begin{equation*}
\rk(\tilde{E_k}/F)\leq [F:\mathbb{Q}] \quad \text{for}\ k\in\{1,2\}
\end{equation*}
for all subfields $F$ of $K$ and for all $\tilde{E_k}/F$ isomorphic to $E_k$ over $K$.
\end{assumption}
Possibly after reordering the $g_{ij}^{E_k}$, we can then find $$\alpha_{ij}^{\chi,E_k}, \quad 0\leq i\leq j\leq \rk(E_k/K)-1,$$  such that
\begin{equation}
\label{eq:def_alphaij}
h^{\chi, E_k} = \sum_{i,j}\alpha_{ij}^{\chi,E_k}g_{ij}^{E_k}.\
\end{equation}
By abuse of notation, we also write $h^{\chi, E_k}(P)$ for $h^{\chi, E_k}(P,P)$ and similarly for $g_{ij}^{E_k}$. Embed $E_k(K)$ in $E_k(K\otimes \Q_p)$ via $\sigma$.
Note that, while $h^{\chi, E_k}$ only makes sense as a function of $E_k(K)$, each $g_{ij}^{E_k}$ is the restriction of a function on $E_k(K\otimes\Q_p)$, which we also denote by $g_{ij}^{E_k}$. Furthermore, if $\zvec\in E_k(K\otimes \mathbb{Q}_p)$, we write $h^{\chi, E_k}_{\mathfrak{p}_j}(\zvec)$ for $h^{\chi, E_k}_{\mathfrak{p}_j}(z_j)$, which is locally analytic away from the disk of the point at infinity $\infty_{E_k}$. 

Let $Q_1 = (0,\sqrt{a_0})\in E_1(K(\sqrt{a_0}))$, $Q_2 = (0,a_0)\in E_2(K)$ and, for a point $P\in X(K_j)$, denote by $]P[$ the residue disk centered at $P$ in $X(K_j)$. 
For $k=1,2$, define
\begin{align*}
 X^{(k)}(K\otimes \Q_p)&=\prod_{j=1}^m \left(X(K_{j})\setminus \left(]\sigma_j(\varphi_k^{-1}(Q_k))[\ \cup\ ]\sigma_j(\varphi_k^{-1}(-Q_k))[\right)\right)\\
  X^{(k)}(K) &=\sigma(X(K))\cap X^{(k)}(K\otimes \Q_p),
 \end{align*}
 where, if $Q_1$ is not defined over $K$, we let $\sigma_j(\varphi_1^{-1}(\pm Q_1)) = (0,\pm \sqrt{\sigma_j(a_0)})$ if $a_0\in \OO_j^2$ and $]\sigma_j(\varphi_1^{-1}(\pm Q_1))[ = \emptyset$ otherwise.

If $a_0$ is not a square in $K$, we need to extend the local and global heights on $E_1/K$ to $E_1/L$ where $L=K(\sqrt{a_0})$.  Let $\chi^{\prime}$ be the continuous idele class character of
$L$ defined by composing the trace maps of $\chi$ with the field traces $L_{\mathfrak{p}^{\prime}}/K_{\mathfrak{p}_j}$ at each $\mathfrak{p}^{\prime}\mid \mathfrak{p}_j$ and normalised so that it restricts to $2\chi$ on $\mathbb{A}_K^\ast/K^\ast$. Write $$h^{\chi,E_1}(Q_1)=\frac{1}{2}h^{\chi',E_1}(Q_1)$$ and, for a prime $\fq$ of $K$, $$h^{\chi,E_1}_{\fq}(Q_1)=\frac{1}{[L_{\fq^{\prime}}:K_{\fq}]}h_{\fq^{\prime}}^{\chi',E_1}(Q_1),$$ where $\fq^{\prime}$ is any prime of  $L$ above $\fq$. Since $\mathrm{Gal}(L/K)$ acts on $Q_1$ by multiplication by $\pm 1$ and by definition of $\chi^{\prime}$, the quantity $h^{\chi,E_1}_{\fq}(Q_1)$ is well-defined and all properties of $h^{\chi,E_1}_{\fq}$ for points over $K$ hold for $Q_1$.

The following result is a version of Theorem~\ref{biell_intro} which is computationally more useful, since the function of Theorem~\ref{biell_intro} is defined only on a proper subset of $X^{(k)}(K\otimes \Q_p)$ for each $k$.
\begin{theorem}
\label{thm:bielliptic}
For each $k\in\{1,2\}$, there exists an explicitly computable finite set $T^{\chi,k}\subset \Q_p$ such that the locally analytic function $\rho^{\chi,k}\colon X^{(k)}(K\otimes \Q_p)\to \Q_p$, defined by
\begin{align*}
\rho^{\chi,k}(\zvec) \colonequals  \sum_{j=1}^m\left(2h_{\fp_j}^{\chi,E_{3-k}}(\varphi_{3-k}(z_j))-h^{\chi, E_k}_{\fp_j}(\varphi_k(z_j)+Q_k)-h^{\chi,E_k}_{\fp_j}(\varphi_k(z_j)-Q_k)\right)\\
-2\sum_{i,j}\alpha_{ij}^{\chi,E_{3-k}}g_{ij}^{E_{3-k}}(\varphi_{3-k}(\zvec))+2\sum_{i,j}\alpha_{ij}^{\chi,E_k}g_{ij}^{E_{k}}(\varphi_k(\zvec))+2h^{\chi,E_k}(Q_k),
\end{align*}
takes values in $T^{\chi,k}$ when restricted to $X^{(k)}(K)$.
\end{theorem}

\begin{proof}
Let $\fq$ be any prime of $K$ and $\ell\in\{1,2\}$. By \cite[Lemma 7.4]{BD18}, the local height $h_{\fq}^{\chi, E_{\ell}}$ satisfies the quasi-parallelogram law, i.e.\ for all points $P,R\in E_{\ell}(K_\fq)$ such that $P,R,P\pm R\neq \infty_{E_{\ell}}$, we have
\begin{equation}
\label{eq:quasi_par}
h_{\fq}^{\chi, E_{\ell}}(P+R)+h_{\fq}^{\chi, E_{\ell}}(P-R)=2h_{\fq}^{\chi, E_{\ell}}(P)+2h_{\fq}^{\chi, E_{\ell}}(R)-2\chi_{\fq}(x(R)-x(P)).
\end{equation}
This, together with (\ref{eq:def_alphaij}), implies that $\rho^{\chi,k}(X^{(k)}(K))$ is contained in 
\begin{align*}
T^{\chi,k} \colonequals  \biggl\{\sum_{\fq\nmid p}\left(h^{\chi,E_k}_{\fq}(\varphi_k(z_{\fq})+Q_k)+h^{\chi,E_k}_{\mathfrak{q}}(\varphi_k(z_{\fq})-Q_k)-2h^{\chi, E_{3-k}}_{\fq}(\varphi_{3-k}(z_{\fq}))\right):&\\
(z_{\fq})\in \prod_{\fq\nmid p}X(K_{\fq})\setminus \{\varphi_k^{-1}(\pm Q_k)\}&\biggr\}.
\end{align*}
The elementary proof of finiteness of $T^{\chi,k}$ given in \cite[Proposition 6.5]{bia:QCKim} when $K=\mathbb{Q}$ and $\chi$ is the cyclotomic character uses properties of the local heights away from $p$ that hold also in this more general setting. Any detail that is omitted here can thus be deduced from \cite{bia:QCKim}. Let $\mathfrak{q}\nmid p$ and define
\begin{equation*}
W_{\mathfrak{q}}^{\chi,E_{\ell}} = \{h_{\mathfrak{q}}^{\chi,E_{\ell}}(P)\,:\,P\in E_{\ell}(K_{\mathfrak{q}}), x(P)\in \OO_{\mathfrak{q}}\}.
\end{equation*}
As in the proof of Theorem \ref{T:rhos}, $W_{\mathfrak{q}}^{\chi,E_{\ell}}$ is finite and identically zero for almost all $\mathfrak{q}$. On the other hand, if $P\in E_{\ell}(K_{\fq})$ with $x(P)\not\in \OO_{\fq}$, then
\begin{equation}
\label{eq:height_non_int}
h_{\fq}^{\chi,E_{\ell}}(P) = \chi_{\fq}(x(P)).
\end{equation}

Let $z\in X(K_{\fq})\setminus \{\varphi_k^{-1}(\pm Q_k)\}$ and define
\begin{equation*}
w_{\fq}^{\chi,E_{k}}(z) =h^{\chi,E_k}_{\fq}(\varphi_k(z)+Q_k)+h^{\chi,E_k}_{\mathfrak{q}}(\varphi_k(z)-Q_k)-2h^{\chi, E_{3-k}}_{\fq}(\varphi_{3-k}(z)).
\end{equation*}
If $\varphi_k(z)=\infty_{E_k}$, then $w_{\fq}^{\chi,E_{k}}(z)= 2h^{\chi,E_k}_{\fq}(Q_k)-2h^{\chi,E_{3-k}}_{\fq}(Q_{3-k})$; otherwise, by (\ref{eq:quasi_par}), we have
\begin{align*}
w_{\fq}^{\chi,E_{k}}(z)= 2\left(h^{\chi,E_k}_{\fq}(\varphi_k(z))+h^{\chi,E_k}_{\fq}(Q_k)-\chi_{\fq}(x(\varphi_k(z)))-h^{\chi, E_{3-k}}_{\fq}(\varphi_{3-k}(z))\right)\\
 = 2\left(h^{\chi,E_k}_{\fq}(\varphi_k(z))+h^{\chi,E_k}_{\fq}(Q_k)+\chi_{\fq}(x(\varphi_{3-k}(z)))-\chi_{\fq}(a_0)-h^{\chi, E_{3-k}}_{\fq}(\varphi_{3-k}(z))\right).
\end{align*}
If $0\leq \ord_{\fq}(x(z))\leq \ord_{\fq}(a_0)/2$, then both $x(\varphi_1(z))$ and $x(\varphi_2(z))$ are integral and
\begin{equation*}
\chi_{\fq}(x(\varphi_1(z)))\in \{n\chi_{\fq}(\pi_{\fq}): 0\leq n\leq \ord_{\fq}(a_0)\}.
\end{equation*}
 In the remaining cases, exactly one of $x(\varphi_{1}(z))$ and $x(\varphi_{2}(z))$ is non-integral. Thus, by (\ref{eq:height_non_int}) and finiteness of $W_{\fq}^{\chi,E_{\ell}}$, we deduce that $w_{\fq}^{\chi,E_{k}}(z)$ takes values in a finite set $T_{\fq}^{\chi,k}$. Furthermore, since $Q_{\ell}$ is an integral point and $a_0$ divides the discriminant of $E_2$, $T_{\fq}^{\chi,k}\subseteq \{0\}$ for all primes $\fq$ that divide neither the discriminant of $E_1$ nor the one of $E_2$.
\end{proof}

Assuming Leopoldt's conjecture, we have $r_2+1$ choices of independent characters $\chi$ and then Theorem \ref{thm:bielliptic} gives at most $r_2+1$ independent locally analytic functions on each $X^{(k)}(K\otimes \Q_p)$ vanishing on the global points $X^{(k)}(K)$. As in Section \ref{S:qcnf}, we construct  $r_K$ other functions using the relations among the functionals $f_i^{J}$ imposed by our assumption (\ref{eq:rkgenusbielliptic}). In practice, each such relation can be reduced to a relation among the $f_i^{E_k}$ for some $k$.

\begin{remark}
As soon as $m>1$ and $a_0\in \OO_j^2$ for some $j$,  we have $$X^{(1)}(K\otimes\Q_p)\cup X^{(2)}(K\otimes \Q_p)\subsetneq
  X(K\otimes\Q_p).$$ Thus, in order to turn the method outlined in this section into a
  strategy to compute $X(K)$, one also needs to deal with the residue disks that are not
  covered by Theorem~\ref{thm:bielliptic}. If, for instance, there are no $K$-rational
  points in these disks, then we can use the Mordell-Weil sieve to prove this. 
  See also Remark \ref{rmk:twice_over_Q} for a class of examples where we can immediately prove that there exist no $K$-rational points in such disks.
\end{remark}

\begin{remark}
\label{rmk:twice_over_Q}
A quadratic Chabauty computation over number fields can sometimes be replaced with several quadratic Chabauty computations over $\Q$. For example, suppose that $X$ is defined over $\Q$ and that we want to determine $X(K)$, where $K=\Q(\sqrt{d})$ for some square-free $d\in\Q$. Assume that $\rk(E_1/\Q)=\rk(E_2/\Q)=1$, $\rk(E_1/K)=2$ and $E_2(K) =E_2(\Q)$. If $P\in X(K)$, then $\varphi_2(P)\in E_2(K)=E_2(\Q)$ and so $x(P)^2\in \Q$, $y(P)/x(P)\in \Q$. It follows that computing $X(K)$ is equivalent to computing $X(\Q)$ and $X^{\prime}(\Q)$ where
\[
X^{\prime}\colon y^2=x^6+\frac{a_4}{d}x^4+\frac{a_2}{d^2}x^2+\frac{a_0}{d^3}.
\]
The curve $X^{\prime}$ is a genus $2$ bielliptic curve whose corresponding elliptic curves are the quadratic twists $E_1^{d^{-1}}$ and $E_2^{d^{-2}}\cong_{\Q} E_2$. In particular, $\rk(E_1^{d^{-1}}/\Q)=1$ and $\rk(E_2^{d^{-2}}/\Q)=1$. Quadratic Chabauty can thus be used to determine both $X(\Q)$ and $X^{\prime}(\Q)$.

We also note that, even when $p$ is split in $K$, in this example we have $$\sigma(X(K))= X^{(1)}(K)\cup X^{(2)}(K).$$
Indeed, $X(K\otimes \Q_p)\setminus (X^{(1)}(K\otimes \Q_p)\cup X^{(2)}(K\otimes \Q_p))$ consists of those points $(z_1,z_2)\in X(\Q_p)\times X(\Q_p)$ such that $\overline{z_i}\in \{\overline{(0,\pm \sqrt{a_0})}\}$, $z_j \in \{\overline{\infty^{\pm}}\}$ if $\{i,j\} = \{1,2\}$, where $\sqrt{a_0}$ is a fixed square root of $a_0$ in $\Q_p$. Suppose that $(z_1,z_2)=\sigma(z)$ for some $z\in X(K)$. Then $x(z_1)^2 = x(z_2)^2$ since $\varphi_2(z)\in E_2(\Q)$. In particular, $\ord_p(x(z_1))=\ord_p(x(z_2))$: a contradiction.

\end{remark}

\section{Algorithms and examples for integral points over quadratic fields}\label{S:alg_ex_int}
To illustrate our method, we give detailed algorithms for carrying out the process described in 
Section~\ref{S:qcnf} for elliptic curves over real 
and imaginary quadratic fields and curves of genus 2 over imaginary
quadratic fields. We then apply these algorithms in several examples.
Because of current limitations of our implementation, we restrict to the case where $p$ is 
split in $K$, so that the Coleman integration takes place over $\Q_p$. 
The extension to more general good primes $p$ is straightforward.

We keep the notation of Section~\ref{S:qcnf}.
Let $K$ be a quadratic number field and let $p$ be a prime number such that $p\OO_K = \fp_1\fp_2$ is split.
Let $\chi_1, \ldots,\chi_{r_2+1}\in V_K$ be independent continuous idele class characters.
By Example~\ref{ex:split}, they correspond to
pairs $c^{(k)} =
(c^{(k)}_{1},c^{(k)}_{2}) \in \Q_p^2$ such that 
$h^{\chi_k}_{\fp_j}(z) = c^{(k)}_j\cdot \tau_j(z)$ for all $j \in \{1,2\}$ and $z\in
X(K_j)\setminus\{\infty\}$. 
Writing
\[
  h^{(k)} \colonequals h^{\chi_k} = \sum_{0\le i\le j \le r-1} \alpha^{(k)}_{ij}g_{ij}\,,
\]
we get
\begin{equation}\label{rho_split}
  \rho_k(\zvec) \colonequals\rho^{\chi_k}(\zvec) =  \sum^d_{j=1}c^{(k)}_j\tau_j(z_j) -
  \sum_{i,j}\alpha_{ij}^{(k)}f_i(\zvec)f_j(\zvec)
\end{equation}
for $\zvec \in \UU(\OO_K\otimes\Z_p)$.

Algorithms for the case $K=\Q$ are discussed in great detail in \cite{BBM16}; 
they readily generalize to more general number fields $K$ provided $p$ is totally split.
For instance, it is easy to see that the sets $T^{(k)}$ can be computed in practice 
using the method of~\cite[\S3.4]{BBM16}. 
The $p$-adic heights can be computed using a straightforward adaptation of~\cite[Algorithm
3.8]{BaMuSt12}, and we deduce the constants $\alpha_{ij}^{(k)}$ as in~\cite[\S3.2]{BBM16}.
For the $p$-adic analytic computations we use {\tt Sage}~\cite{sage}, the remaining computations are
done in {\tt Magma}~\cite{magma}.

\begin{remark}\label{R:prec_ana}
Since~\cite[\S3]{BBM16} contains a detailed precision analysis for each of these
algorithms, we deduce that the $p$-adic objects in our results above are
indeed explicitly computable in the sense described in the introduction; they are
provably computable to any desired finite precision.
\end{remark}

We want to construct two functions $\rho_1,\rho_2$ that map
\[\sigma(\UU(\OO_K)) \subset \UU(\OO_K\otimes \Z_p)\]
into an explicitly computable finite set $T^{(1)}$ ($T^{(2)}$, respectively)
and have a finite common solution set. 
Once we have done this, we have to check if each solution $\zvec=(z_1, z_2)$ 
actually comes from a $K$-rational point, that is, whether $z_1 = \sigma_1(Q)$ and $z_2 =
\sigma_2(Q)$ for some $Q \in \UU(\OO_K)$. 
We can use, for instance, that if this is the case, then we must have $x(z_1)x(z_2) \in \Q$ 
and $x(z_1) + x(z_2) \in \Q$.

The functions $f^{(j)}_{i}\colon \UU(\OO_K\otimes \Z_p)\to \Q_p$ are given by
\[
  f^{(j)}_i(\zvec) = \int^{z_j}_{\infty} \omega_i\,;
\]
the corresponding linear functionals in $(J(K)\otimes\Q_p)^\vee$ will also be denoted by 
$  f^{(j)}_i$.

\subsection{Elliptic curves over real quadratic fields}\label{sec:real}
Let 
$E/K$ 
denote an elliptic curve over a real quadratic field\footnote{\label{ft:also_imaginary}The algorithm described in this subsection
works also when $K$ is imaginary; however, while in the real case the rank assumption here
is optimal, in the imaginary case we can work with curves of larger rank, as is explained
in \S \ref{sec:ell_imag}.} $K$ given by a Weierstrass equation\footnote{The setup in the
previous section required an equation of the form $y^2=f(x)$, but for elliptic curves we can
extend it easily to cover general Weierstrass equations.}
   with good reduction at
both $\fp_1$ and $\fp_2$. 
We set $f_0 \colonequals f^{(1)}_0$ and $f_1 \colonequals f^{(2)}_0$.
Equation~\eqref{ranks} is satisfied if and only if $r =\rk(E/K)=0$ or $1$; we assume that $r=1$.
It follows that there is a linear dependence $f_0=bf_1$
between the functionals $f_0$ and $f_1$ on $E(K)\otimes \Q_p$ 
The first function $\rho_1$ is constructed using the cyclotomic $p$-adic height pairing
$\hc$. By~\eqref{rho_split}, we have 
\[
  \rho_1(\zvec)  \colonequals \rc(\zvec) =  \tau_1(z_1)+\tau_2(z_2)
  -\alpha f_0(\zvec)^2,
\]
where 
\[
  \hc = \alpha g_{00} = \alpha f_0^2\,,
\]
that is, $\alpha = \alpha^{(1)}_{00}$.
We can compute $\alpha$ by comparing $\hc(Q)$ and $f_0^2(Q)$ for some nontorsion point $Q
\in E(K)$.
Letting $T  \colonequals  T^{(1)}$ denote the finite set of values that $-\sum_{\fq \nmid p} \hc_{\fq}(z_{\fq}-\infty)$ can
take for $(z_{\fq}) \in \prod_{\fq\nmid p}\UU(\OO_{K_\fq})$,
we then must have
\[
  \rho_1(\sigma_1(Q),\sigma_2(Q)) \in T
\]
for all $\OO_K$-integral points $Q \in E(K)$.

The second function comes from the dependence between $f_0$ and $f_1$ and is given by
\[
  \rho_2(\zvec) = 
\int^{z_1}_\infty\omega_0 -b\int^{z_2}_\infty\omega_0\,;   
\]
we have $ \rho_2(\sigma_1(Q),\sigma_2(Q)) =0$  for all $Q \in E(K)$.
We can use any nontorsion $Q \in E(K)$ and compare $f_0(Q)$ to $f_1(Q)$ to find $b$.

Note that, if $E$ is defined over $\Q$ and $w\in T$, we can always reduce the system 
\begin{equation*}
\begin{cases}
\rho_1(z_1,z_2)=w\\
\rho_2(z_1,z_2)=0
\end{cases}
\end{equation*}
to one equation in one variable. Indeed, first observe that
\begin{itemize}
\item if $\rk(E/\Q)=\rk(E/K)=1$, then for every $Q\in E(K)$ there exists $n$ such that $nQ\in E(\mathbb{Q})$. Thus, by linearity, $f_0(Q)=f_1(Q)$;
\item if $\rk(E/\Q)=0$, then for every $Q\in E(K)$ we have $Q+Q^{c}\in E(\mathbb{Q}) = E(\mathbb{Q})_{\mathrm{tors}}$, where $Q^c$ is the Galois conjugate of $Q$. We deduce that $f_0(Q)+f_1(Q)=f_0(Q+Q^c)=0$.
\end{itemize}
Furthermore,
\begin{equation*}
\int_{\infty}^{z}\omega_0=\int_{\infty}^{z'}\omega_0\iff z-z'\in E(\Q_p)_{\mathrm{tors}}.
\end{equation*}
Suppose that $\rk(E/\Q)=1$ and that we want to find all $Q\in \UU(\OO_K)$ such that $(\overline{\sigma_1(Q)},\overline{\sigma_2(Q)})=(\overline{P_1},\overline{P_2})\in \overline{E}(\F_p)\times \overline{E}(\F_p)$. If there are no torsion points in $E(\Q_p)$ reducing to $\overline{P_1}-\overline{P_2}$, then the above shows that no such $Q$ exists. Otherwise, we solve
\[
\rho_1((x(s),y(s))
, (x(s),y(s))+R) =w\]
where $(x(s),y(s))$ is a parametrization for the points in $E(\Q_p)$ reducing to $\overline{P_1}$ and $R$ is the unique torsion point\footnote{Since $p$ is
odd, the group $E_1(\mathbb{Q}_p)$ of $\mathbb{Q}_p$-points on the formal group at $p$ is torsion-free. Therefore, $E(\mathbb{Q}_p)_{\mathrm{tors}}$ can be realized as a subgroup of $\overline{E}(\mathbb{F}_p)$.} in $E(\mathbb{Q}_p)$ such that $\overline{P_2}-\overline{P_1}=\overline{R}$ (cf. Example \ref{ex:E_real_over_Q}). The case $\rk(E/\Q)=0$ is very similar.

\begin{example}
\label{ex:E_over_K}
Consider the elliptic curve \href{http://www.lmfdb.org/EllipticCurve/2.2.5.1/199.1/c/1}{199.1-c1} \cite{lmfdb} over $K=\Q(a)$, where $a$ is a root of $x^2 - x - 1$. It admits the global minimal model
\begin{equation*}
E\colon y^2+y=x^3+(a+1)x^2+ax.
\end{equation*}
The prime $p=11$ is the smallest prime splitting in $K$. The curve has everywhere good reduction, except at $(-3a+16)$, where it has Kodaira symbol $\mathrm{I}_1$: thus, $T=\{0\}$. The common zero set of $\rho_1$ and $\rho_2$ contains the integral points
\begin{align*}
\pm (-1,0), \pm (0,0), \pm (a,2a), \pm (-2a+3,4a-7), \pm (a+1,3a+1), \pm (-a,0),\\
 \pm (42a + 27, -420a - 259 ), \pm(-6a+8,-18a+29), \pm (-a + 1, a - 2)
\end{align*}
and $138$ other pairs of points in $\sigma_1(E)(\mathbb{Q}_p)\times \sigma_2(E)(\mathbb{Q}_p)$, where $\sigma_i(E)/\Q_p$ is the curve obtained from the Weierstrass equation for $E$ via the embedding $\sigma_i\colon K\hookrightarrow \Q_p$.
\end{example}

  \begin{example}
\label{ex:E_real_over_Q}
Consider the elliptic curve with LMFDB label \href{http://www.lmfdb.org/EllipticCurve/Q/192a1 /}{192.a3} \cite{lmfdb}
$$E \colon  y^2 = x^3 - x^2 - 4x - 2$$ and let $K = \Q(\sqrt{3})$. We have that $\rk (E/K)
    = \rk (E/\Q) = 1$ and that $p = 13$ is of good ordinary reduction for $E$ and is split in $K$.
As explained above, in this setting we can reduce our computations to finding solutions to power series equations in one variable. Note in particular that $\#\overline{E}(\mathbb{F}_p)=12\not\equiv 0 \bmod{p}$, so each residue disk contains precisely one torsion point of $E(\Q_p)$.

The only inputs in this computation which come from our specific choice of quadratic field
    $K$ are the reduction type of $E$ at the bad primes (and consequently the set $T$) and
    the choice of prime $p$, since we require that $p$ is split.

For this reason, it is not surprising that, besides recovering the integral points $$\pm(3,2), (-1,0), ( 1\pm \sqrt{3} , 0), \pm(25 \pm
15\sqrt{3}, 180 \pm 104\sqrt{3})\,,$$
we also recognize in our zero set the points
\begin{equation*}
\pm \left(\frac{-1\pm \sqrt{-3}}{2},  \frac{-3\pm\sqrt{-3}}{2}\right), \pm(-1\pm i, 2)\,,
\end{equation*}
all of which belong to the $2$-saturation of $E(\Q)$.
\end{example}

\subsection{Elliptic curves over imaginary quadratic fields}\label{sec:ell_imag}
The following brief discussion on elliptic curves over imaginary quadratic fields will be useful for the computation of rational points on bielliptic curves over imaginary quadratic fields of \S \ref{sec:rank_4_rational}.
Let $K$ be an imaginary quadratic number field and let $E$ be an elliptic curve defined over $K$ with $r=\rank(E/K)=2$, such that $E$ has good reduction at
both $\fp_1$ and $\fp_2$.
Suppose that $f_0$ and $f_1$
are linearly independent. Note that:
\begin{lemma}
\label{lemma:indep}
Suppose that $E$ as above can be defined over $\Q$. Then $f_0$ and $f_1$ are linearly independent if and only $\rk(E/\Q)=1$.
\end{lemma}
\begin{proof}
If $\rk(E/\Q)=r$ (resp.\ $\rk(E/\Q)=0$), then $f_0=f_1$ (resp.\ $f_0=-f_1$) on $E(K)$ (cf.\ also the discussion before Example \ref{ex:E_over_K}). Conversely, suppose that $\rk(E/\Q)=1$ and that $f_0$ and $f_1$ are linearly dependent on $E(K)\otimes \Q_p$. Since neither of $f_0$ and $f_1$ is identically zero on $E(K)$ (as $E(K)$ contains points of infinite order), then there exists $\alpha\in\mathbb{Q}_p$ such that $f_0=\alpha f_1$. Let $P\in E(\mathbb{Q})$ of infinite order. Then $f_0(P)=f_1(P)$ and hence $\alpha=1$. Thus, for all $Q\in E(K)$, we have $\int_{\infty}^{\sigma_1(Q)}\omega_0 = \int_{\infty}^{\sigma_2(Q)}\omega_0$, so that $\sigma_1(Q-Q^{c})\in E(\mathbb{Q}_p)_{\mathrm{tors}}$, where $Q^c$ is the Galois conjugate of $Q$. Let $R\in E(K)_{\mathrm{tors}}$ such that $Q=Q^{c}+R$ and let $n$ be the additive order of $R$. Then $nQ = nQ^{c}$, i.e.\ $nQ\in E(\mathbb{Q})$, contradicting $\rk(E/K)>\rk(E/\Q)$.
\end{proof}

Under our assumption, the bilinear forms $g_{ij}$ ($0 \le i\le j\le 1$) generate $\Bil$, the space of $\Q_p$-valued bilinear forms on $E(K)\otimes \Q_p$.
Hence the $p$-adic heights $\hc$ and $\ha$ can be expressed as linear combinations 
\[
\hc =\sum \alpha^{\cyc}_{ij}g_{ij}\quad \text{and} \quad \ha =\sum \alpha^{\anti}_{ij} g_{ij}.\]
Using independent nontorsion points $P,Q\in E(K)$, we can find all $\alpha^{\cyc}_{ij}$ (resp.\
$\alpha^{\anti}_{ij}$) by
computing $f_0(P), f_1(P), f_0(Q), f_1(Q)$ and  $\hc(P), \hc(Q), \hc(P+Q)$ (resp.\ $\ha(P),
\ha(Q), \ha(P+Q)$). 

When $E$ is defined over $\Q$, we have $\alpha_{00}^{\cyc} =\alpha_{11}^{\cyc}$ and $\alpha_{01}^{\anti}=0$, $\alpha_{00}^{\anti}=-\alpha_{11}^{\anti}$. Indeed, for all $R_1,R_2\in E(K)$, we have $\hc(R_1,R_2)=\hc(R_1^c,R_2^c)$ and hence
\begin{equation*}
(\alpha^{\cyc}_{00}-\alpha^{\cyc}_{11})(g_{00}(R_1,R_2)-g_{11}(R_1,R_2)) =0;
\end{equation*}
the first claim then follows since the $g_{ij}$ are a basis for $\Bil$. Similarly, the anticyclotomic height pairing changes of sign if we first act by Galois, proving the second claim.

\subsection{Genus~2 curves, imaginary quadratic fields}\label{sec:g2imag}
Suppose that $X \colon y^2=f(x)$ has genus~2 and rank~4 over an imaginary quadratic field $K$.
Let $p$ be a split prime of good, ordinary reduction and let $\sigma_1,\sigma_2 \colon K
\hookrightarrow \Q_p$ denote the two embeddings corresponding to the primes $\fp_1,\fp_2$ of $K$ above $p$.
Then we have~4 functionals $f^{(j)}_i \colon J(K)\otimes_{\Z}\Q_p\to \Q_p$, extended linearly from the
functions 
  \[
    f^{(j)}_i(Q) = \int^{\sigma_j(Q)}_\infty \omega_i,\;
  \]
on $X(K)$,
which we assume to be independent.
For simplicity, we renumber them 
\[
  f_0 =f_0^{(1)},\,f_1 = f_0^{(2)},\, f_2 = f_1^{(1)},\,f_3 = f_1^{(2)}\,.
\]
As in the proof of Theorem~\ref{T:rhos}, the functions $g_{ij}(P,Q) =
\frac{1}{2}(f_i(P)f_j(Q)+f_i(Q)f_j(P))$, for $0\le i\le j \le 3$ 
form a basis of the bilinear forms on $J(K)\otimes_{\Z}\Q_p$.
We can express the $p$-adic height pairings $\hc$ and $\ha$ in terms of the $g_{ij}$:
\[
  \hc  = \sum_{i,j} \alpha^{\cyc}_{ij} g_{ij}\;\textrm{ and }\;
  \ha  = \sum_{i,j} \alpha^{\anti}_{ij} g_{ij}.
\]
In order to compute the constants $\alpha^{\cyc}_{ij}$ and $\alpha^{\anti}_{ij}$, we find divisors $D_1,\ldots,D_4 \in \Div^0(X)$ such that the
corresponding points generate a finite index subgroup of $J(K)$ modulo torsion.
Then we have, for each pair $(k,l)$, the
relationships $$\hc(D_k, D_l) = \sum_{0 \leq i \leq j \leq 3} \alpha^{\cyc}_{ij} \cdot
  \left(\frac{1}{2}\left(f_i(D_k)f_j(D_l) + f_j(D_k)f_i(D_l)\right)\right) $$ and  
  $$\ha(D_k, D_l) = \sum_{0 \leq i \leq j \leq 3} \alpha^{\anti}_{ij} \cdot
  \left(\frac{1}{2}\left(f_i(D_k)f_j(D_l) + f_j(D_k)f_i(D_l)\right)\right)\,.$$  
  That is, we can solve the linear system relating the height pairings $\hc$ and $\ha$ to the
  natural basis $(g_{ij} \colon 0 \le i \le j \le 3)$ to recover the $\alpha^{\cyc}_{ij}$ and the
  $\alpha^{\anti}_{ij}$.

Every $\OO_K$-integral point $Q$ on $X$ must satisfy
\[
  \rc(\sigma_1(Q), \sigma_2(Q))\in T^{\cyc}\,;\quad \ra(\sigma_1(Q), \sigma_2(Q) \in
  T^{\anti}\,,
\]
where 
\begin{align*}
  \rc(\zvec) &=   \tau_1(z_1) + \tau_2(z_2)-\sum_{ij} \alpha^{\cyc}_{ij} f_i(\zvec)f_j(\zvec)\\
  \ra(\zvec) &= 
  \tau_1(z_1) - \tau_2(z_2) - \sum_{ij} \alpha^{\anti}_{ij} f_i(\zvec)f_j(\zvec)
\end{align*}
and where $T^\cyc$ and $T^\anti$ are the respective possible values for the sum of the local heights away from $p$
of an $\OO_K$-integral point on $X$.

\begin{example}\label{int_ex}
We consider the curve  $$X  \colon  y^2 = x^5 - x^4 + x^3 + x^2 - 2x + 1$$ over $K = \Q(\sqrt{-3})$ and note that $\rk (J/\Q) = 2$ and $\rk (J/K) = 4.$    Let \begin{align*}&P_1 = (0,1), Q_1 = \left(\frac{1}{2}(-\sqrt{-3} + 1), \frac{1}{2}(\sqrt{-3} + 1)\right), P_2 = (1,1), Q_2 = (-1,1),\\
&P_3 = (\sqrt{-3}, 2\sqrt{-3}+1), Q_3 = (2,5), P_4 = (\frac{1}{2}(\sqrt{-3} + 1), \frac{1}{2}(-\sqrt{-3} + 1)), Q_4 = (4, 29),\end{align*}
and let $D_i= (P_i) - (Q_i)$ for $i  =1, \ldots, 4$. The points $[D_1], [D_2], [D_3],
[D_4]$ generate a finite index subgroup $G$ of $J(K)$.  The smallest good ordinary split prime for this curve is $p = 7$.

We recover the following $\OO_K$-integral points on $X$: 
\begin{align*}
&(-1,\pm 1), (\sqrt{-3}, \pm (2\sqrt{-3} + 1)), (0, \pm 1), (2,\pm 5),(4,\pm 29),\\
&(\sqrt{-3} + 1, \pm (-\sqrt{-3} + 4)), (-\sqrt{-3}, \pm (-2\sqrt{-3} + 1)),(1,\pm 1),\\
&(-\sqrt{-3} + 1, \pm (\sqrt{-3}+4)), \left(-\frac{1}{2}\sqrt{-3} + \frac{1}{2}, \pm \left(\frac{1}{2}\sqrt{-3} + \frac{1}{2}\right)\right), \\
&\left(\frac{1}{2}\sqrt{-3} + \frac{1}{2}, \pm \left(-\frac{1}{2}\sqrt{-3} + \frac{1}{2}\right)\right).
\end{align*}
We apply the Mordell-Weil sieve to show that these are
indeed the only integral points. To this end, we apply quadratic Chabauty for the primes~7
and~19. The subgroup $G$ is not saturated at~3, so we use the modified Mordell-Weil sieve
as described in~\cite[Appendix~B]{BD18}. In the notation given there, we used the
auxiliary integer $N=10$ and 
20 primes of $K$ to rule out all but one pair of residue classes modulo $7^4$ and $19^3$,
  respectively; the largest norm of a prime used in this computation was $4591$. In order
  to show that the remaining pair does not correspond to an integral point either, we used
  the auxiliary integer $N=8$.
\end{example}

\begin{remark}\label{R:only_quadratic}
  It appears to be difficult to find examples over number fields $K$ of degree $>2$ for
  which the computation of integral points using our method is applicable in practice. Some restrictions are due to external components of our
  algorithms. For instance, we need to
  compute the Mordell-Weil rank $r$, find $r$ independent points in $J(K)$, compute regular models and
  intersections over $K$ (the {\tt Magma}-command {\tt RegularModel} uses 
  arithmetic in $K$, not in a completion) and apply the Mordell-Weil sieve over $K$.
  These computations often fail or are prohibitively expensive for number fields of degree
  $>2$.

  Therefore our current implementation is also limited to quadratic fields. 
  To go beyond this, we would have to algorithmically compute a basis of the space
  of continuous idele class characters, extend our implementation of the root finding methods discussed in
  Appendix~\ref{appendix:roots_multi}. Moreover, the Coleman integration and $p$-adic
  heights routines, currently restricted to $\Q_p$, would have to be extended to finite
  extensions to be able to deal with  
  generators of (a finite index subgroup of) $J(K)/{\mathrm{tors}}$ that do not split into
  sums of points over $\Q_p$. 
\end{remark}

\section{Examples for rational points}
\label{S:examples_rational}
In this section we present examples for Section \ref{sec:rat}, of which we retain the
notation, when $K$ is a quadratic field \footnote{We focus on examples over quadratic fields for the
same reasons as discussed in Remark~\ref{R:only_quadratic}.}. To simplify the exposition, we further assume that $X$ is defined over $\Q$. We let $p$ be an odd prime which splits in $K$ and such that $X$ has good reduction at each prime above $p$. For simplicity we assume that $a_0$ is a square in $K$.

\subsection{Rank 4 over imaginary quadratic fields}
\label{sec:rank_4_rational}
Let $X/\Q$ be as above and let $K$ be an imaginary quadratic field. If $\rk(E_k/\Q)=1$ and $\rk(E_k/K)=2$ for each $k$, then \eqref{eq:rkgenusbielliptic} and Assumption \ref{ass:E1E2} are both satisfied  in view of Lemma \ref{lemma:indep}. Theorem \ref{thm:bielliptic} with $\chi =\chi^{\cyc}$ and $\chi=\chi^{\anti}$ provides us, for each $k\in \{1,2\}$, with two locally analytic functions $\rho^{\cyc,k}$ and $\rho^{\anti,k}$ on $X^{(k)}(K\otimes \Q_p)$ and finite sets $T^{\cyc,k}$ and $T^{\anti,k}$. Varying $k$ and looking at the intersection $B^k$ of the loci $\rho^{\cyc,k}(\zvec)\in T^{\cyc,k}$ and $\rho^{\anti,k}(\zvec)\in T^{\anti,k}$, we obtain a set of points containing all the points in $\sigma(X(K))$ that do not map to $(\overline{\infty^{\pm}}, \overline{(0,\pm\sqrt{a_0}}))$ or to $(\overline{(0,\pm\sqrt{a_0})}, \overline{\infty^{\pm}})$ under reduction to $\overline{X}(\F_p)\times \overline{X}(\F_p)$.

As already mentioned, the discussion of \S \ref{sec:ell_imag} is used in the computation of the constants of Theorem \ref{thm:bielliptic}. To $p$-adically approximate the set $B^k$, we need to compute roots of systems of two equations in two variables. We do so using Algorithm \ref{alg}, Lemma \ref{lemma:double_roots} and the discussion in Appendix \ref{appendix:roots_multi}. Note that $\rho^{\cyc,k}$ and $\rho^{\anti,k}$ are both invariant under the automorphisms $\theta\colon (x,y)\mapsto (\pm x,\pm y)$ of $X$. 
Thus, from the points in $B^k$ reducing to $(\ol{P},\ol{Q})\in \ol{X}(\F_p)\times \ol{X}(\F_p)$ we can deduce the points in $B^k$ reducing to $(\theta(\ol{P}),\theta(\ol{Q}))$. This cuts down the number of pairs of residue disks that we need to consider.

\begin{example}\label{91_ex}
Consider
\begin{equation*}
X = X_0(91)^{+}:y^2 = x^6 - 3x^4 + 19x^2 - 1.
\end{equation*}
The associated elliptic curves $E_1$ (\cite[\href{http://www.lmfdb.org/EllipticCurve/Q/91/a/1}{91.a1}]{lmfdb}) and $E_2$ (\cite[\href{http://www.lmfdb.org/EllipticCurve/Q/91/b/2}{91.b2}]{lmfdb}) each have rank $1$ over $\Q$ and applying quadratic Chabauty with $p=5$ recovers the rational points
  \begin{equation}\label{x091Q-pts}
\{\infty^{\pm},(1,\pm 4),(-1,\pm 4),(1/3,\pm 28/27),(-1/3,\pm 28/27)\}
\end{equation}
together with $2$ other $5$-adic points up to automorphisms. We now carry out quadratic
  Chabauty for $p=5$ and $K = \mathbb{Q}(i)$, over which the elliptic curves $E_1$ and $E_2$ attain rank $2$. Denote by $T^{\cyc,k}$ and $T^{\anti,k}$ the explicitly computable finite sets of Theorem \ref{thm:bielliptic} for $\chi=\chi^{\cyc}$ and $\chi = \chi^{\anti}$, respectively.

\begin{claim}
$T^{\cyc,1}=T^{\mathrm{cyc},2}=T^{\anti,1}=T^{\anti,2}=\{0\}$.
\end{claim}
\begin{proof}
$X_0(91)^{+}$ has 
bad reduction at $7$ and $13$. Furthermore, $E_1$ and $E_2$ are minimal at the primes of $\mathbb{Q}(i)$ above $7$ and $13$ and the reduction type at all these primes is $\mathrm{I}_1$.
\end{proof}
  Besides the $\Q$-rational points listed in~\eqref{x091Q-pts}, in $B^1\cup B^2$ we also recover the following points defined over $\mathbb{Q}(i)$:
\begin{equation}\label{x091Qi-pts}
(0, \pm i), (\pm(-2i + 1), \pm(-8i + 12)), (\pm(2i + 1), \pm (8i+ 12)),
\end{equation}
as well as $20$ other $5$-adic points up to automorphisms $(x,y)\mapsto (\pm x,\pm y)$.
  We apply the Mordell-Weil sieve to show that these $p$-adic points do not
  correspond to $\Q(i)$-rational points and that there are no $\Q(i)$-rational points
  mapping to the pairs of points $$
  (\overline{\infty^{\pm}}, \overline{(0,\pm\sqrt{-1}})),
  (\overline{(0,\pm\sqrt{-1})}, \overline{\infty^{\pm}}).$$ 
  This proves that all $\Q(i)$-rational points on $X_0(91)^+$ are either listed
  in~\eqref{x091Q-pts} or in~\eqref{x091Qi-pts}.

A technical remark (with reference to Appendix \ref{appendix:roots_multi}): there are four pairs of residue disks in $\overline{X}(\F_p)\times \overline{X}(\F_p)$, up to automorphism, such that a root in $\Z/p\Z\times \Z/p\Z$ of the corresponding system of equations is also a root modulo $p$ of the determinant of the Jacobian matrix. Of these, two correspond to the global points $(0,i)$ and $\infty$ (and thus the determinant of the Jacobian matrix has a zero at those points by Remark \ref{rmk:auto_double_roots}), another one lifts to no root in $\mathbb{Z}/p^n\Z\times \Z/p^n\Z$ for sufficiently large $n$ and for the remaining one we can apply the multivariate Hensel's lemma after lifting the root modulo some small power of $p$ (cf.\ Algorithm \ref{alg}). Furthermore, the strategy outlined after Remark \ref{rmk:auto_double_roots} also shows that there are no other points in the residue disks of $(0,i)$ and $\infty$. Therefore, we conclude that the approximations that we have computed are in one-to-one correspondence with the elements in $B^1\cup B^2$ (up to automorphism). Amongst the $20$ extra points, we recognize the pair of rational points $\zvec=((1,4),(-1,4))$. Indeed, $z_1$ and $z_2$ both map to $(1,4)$ in $E_1(\Q)$ and map to the torsion points $(-1,-4)$ and $(-1,4)$ in $E_2(\Q)$. Thus $g_{ij}^{E_2}(\varphi_2(\zvec))=0$. The claim then follows since the local heights are invariant under pre-multiplication by $-1$.

Finally, note that, since $T^{\anti,k} = \{0\}$, in this case $B^1\cup B^2$ is closed under $(z_1,z_2)\mapsto (z_2,z_1)$. Quotienting by this symmetry, the $20$ extra pairs of points become $14$.

\end{example}
\subsection{Rank 3 over real quadratic fields}
\label{sec:examples_bielliptic_real}
Keep the notation of the beginning of Section \ref{S:examples_rational} and let $K$ now be a real quadratic field\footnote{The strategy also works if $K$ is imaginary. See also footnote \ref{ft:also_imaginary} in \S \ref{sec:real}.}. Then \eqref{eq:rkgenusbielliptic} forces $\rk (J/K)\leq 3$. Since we need Assumption \ref{ass:E1E2} to also hold, we suppose here that $\rk(E_1/K)=2$, $\rk(E_1/\Q)=1$ and $\rk(E_2/K)=1$. This is sufficient by Lemma \ref{lemma:indep} and by the non-degeneracy of the $p$-adic logarithm on an elliptic curve.

For a fixed $k\in\{1,2\}$, Theorem \ref{thm:bielliptic} determines only one function
$\rho_1^k$ and one finite set $T^k$ 
in this case; the second function $\rho_2$ that we need comes from the dependence of $f_0^{E_2}$ and $f_1^{E_2}$ (see also \S\ref{sec:real}). Let
\begin{equation*}
B^k = \{\zvec\in X^{(k)}(K\otimes \Q_p) : \rho_1^k(\zvec)\in T^k, \rho_2(\zvec) = 0\}.
\end{equation*}
Note that 
it suffices to run through $\overline{X}(\mathbb{F}_p)\times \overline{X}(\mathbb{F}_p)/\sim$ where $(P,Q)\sim (R,T)$ if and only if $P=T$ and $Q=R$, as our equations are symmetric with respect to the two embeddings of $K$ into $\mathbb{Q}_p$ (cf.\ the discussion of \S \ref{sec:ell_imag}). As before, we can also use the invariance of our zero sets under the hyperelliptic and bielliptic automorphisms of $X$ to reduce even further the number of pairs of disks that we need to consider.

In view of Remark \ref{rmk:twice_over_Q}, the resulting algorithm is often an alternative to a series of quadratic Chabauty computations over the rationals. This is in particular the case in the following example.
\begin{example}
\label{Wetherell}
Consider
\begin{equation*}
X\colon y^2=x^6+x^2+1.
\end{equation*}
In \cite[Example 6.3]{bia:QCKim} quadratic Chabauty over $\Q$ was used to provide an alternative proof of Wetherell's result \cite{Wet97} that
\begin{equation*}
X(\mathbb{Q})= \{\infty^{\pm},(0,\pm 1 ), (\pm 1/2,\pm 9/8)\}.
\end{equation*}
Let $K=\mathbb{Q}(\sqrt{34})$. We have $\rk(E_1/\mathbb{Q})=1$, $\rk(E_1/K)=2$ and $\rk(E_2/\mathbb{Q})=\rk(E_2/K)=1$. We now determine $X(K)$ exactly, by applying quadratic Chabauty over $K$ with the split prime $p=3$. We may choose
\begin{equation*}
\rho_2(\zvec) = \int_{\infty}^{\varphi_2(z_1)}\omega_0^{E_2}-\int_{\infty}^{\varphi_2(z_2)}\omega_0^{E_2},
\end{equation*}
where $\infty\colonequals \infty_{E_2}$. 
Denote by $T^k$ the explicitly computable finite set of Theorem \ref{thm:bielliptic} for $\chi =\chi^{\cyc}$.
\begin{claim}
\label{claim:wetherell_revisited}
$T^1 =\left\{0,2\log 2,\frac{5}{2}\log 2\right\}$, $T^2 = \left\{-2\log 2,-4\log 2,-\frac{9}{2}\log 2\right\}$.
\end{claim}
\begin{proof}
The elliptic curves $E_1$ and $E_2$ have everywhere good reduction except at the primes above $2$ and $31$. Furthermore, the models
\begin{equation*}
E_1\colon y^2 = x^3+x+1,\qquad E_2\colon y^2 = x^3+x^2+1
\end{equation*}
are globally minimal also over $K$. The prime $31$ is inert in $K$ and $E_1$ and $E_2$ have reduction type $\mathrm{I}_1$ with Tamagawa number $1$ at $31\OO_K$. Thus, the only nontrivial contributions to $T^k$ can come from the unique prime $\mathfrak{q}$ of $K$ above $2$. The curve $E_1$ has Tamagawa number $1$ (and Kodaira symbol $\mathrm{I}_0^{*}$) at $\mathfrak{q}$, so $h_{\fq}^{E_1}$ is identically zero on integral points. The curve $E_2$ has reduction type $\mathrm{I}_1^{*}$ with Tamagawa number $4$ at $\fq$ and we have
\begin{equation*}
h_{\fq}^{E_2}(x,y) =\begin{cases}
0 & \text{ if } |x|_{\mathfrak{q}}=1,\\
-\log 2 \text{ or} -\frac{5}{4}\log 2 &\text{ if } |x|_{\mathfrak{q}}<1
\end{cases} 
\end{equation*}
(see e.g.\ \cite[Proposition 2.4]{bia:QCKim}); in particular, $h_{\mathfrak{q}}^{E_2}(Q_2) =-\log 2$.

For $z\in X(K_{\fq})$, $\varphi_1(z)\neq \pm Q_1$ resp.\ $\varphi_2(z)\neq \pm Q_2$, let
\begin{align*}
w_\fq^{E_1}(z) &= h_{\mathfrak{q}}^{E_1}(\varphi_1(z)+Q_1)+h_{\mathfrak{q}}^{E_1}(\varphi_1(z)-Q_1)-2h_{\mathfrak{q}}^{E_{2}}(\varphi_{2}(z))\\
w_\fq^{E_2}(z) &= h_{\mathfrak{q}}^{E_2}(\varphi_2(z)+Q_2)+h_{\mathfrak{q}}^{E_2}(\varphi_2(z)-Q_2)-2h_{\mathfrak{q}}^{E_{1}}(\varphi_{1}(z)).
\end{align*}
Using the quasi-parallelogram law \eqref{eq:quasi_par}, we find
\begin{itemize}
\item If $|x(z)|_{\fq}\leq 1$, then $w_{\fq}^{E_1}(z)=0$ and $w_{\fq}^{E_2}(z) = -2\log 2$.
\item If $|x(z)|_{\fq}>1$, then $w_{\fq}^{E_1}(z)\in \left\{2\log 2,\frac{5}{2}\log 2\right\}$ and $w_{\fq}^{E_2}(z)\in\left\{-4\log 2,-\frac{9}{2}\log 2\right\}$. 
\end{itemize}
\end{proof}
Since $E_2(K)=E_2(\Q)$ (as $\rank(E_2/K)=\rank(E_2/\Q)$ and $E_2(K)_{\mathrm{tors}} = \{\infty\}$), Remark \ref{rmk:twice_over_Q} tells us that $X(K)\subset X^{(1)}(K\otimes \Q_p)\cup X^{(2)}(K\otimes \Q_p)$ and, in particular, $X(K)\subseteq B^1\cup B^2$. We have
\begin{equation*}
\overline{X}(\F_3) = \{\overline{\infty^{\pm}}, \overline{(0,\pm 1)}, \overline{(\pm 1,0)}\}.
\end{equation*}
It follows from the observations at the beginning of \S \ref{sec:examples_bielliptic_real} that it suffices to consider the points in $X^{(k)}(K\otimes \Q_p)$ reducing to one of the following pairs of residue disks in $\ol{X}(\F_3)\times \ol{X}(\F_3)$:
\begin{align*}
(\overline{\infty^+},\overline{\infty^+}), (\overline{\infty^+},\overline{\infty^-}),(\overline{\infty^+},\ol{(1,0)}) &\text{ with } k=1,\\
 (\ol{(0,1)},\ol{(0,1)}), (\ol{(0,2)},\ol{(0,1)}), (\ol{(0,2)},\ol{(1,0)}), (\ol{(1,0)},\ol{(1,0)}), (\ol{(1,0)},\ol{(2,0)}) &\text{ with } k=2. 
\end{align*}
By Claim \ref{claim:wetherell_revisited}, the intersection of $B^k$ with a residue pair (for $k$ dependent on the pair and chosen as above) is given by the union of the zeros of three systems of two equations in two variables (by a similar argument to the one at the end of Remark \ref{rmk:twice_over_Q}, some pairs could be excluded a priori from containing points in $X(K)$ by valuation considerations, but we leave them all in the computation here). Using the algorithms of Appendix \ref{appendix:roots_multi}, we recover in $B^1\cup B^2$ exactly the points in \textsc{Table} \ref{table:weth_quad} (the residue pairs contained in $X^{(k)}(K\otimes \Q_p)$ that intersect $B^k$ trivially are not listed).
\begin{table}
\begin{center}
\makebox[\textwidth][c]{\begin{tabular}{ |c|c|c|c|c|c|c| } 
\hline
residue pair & $k$ & recovered in $B^{k}$ & if not in $X(K)$, why?  \\
\hline
\multirow{3}{*}{$(\overline{\infty^+},\overline{\infty^+})$} & \multirow{3}{*}{1} & $\infty^{+}$& \\
&  & $\left(\frac{11}{204}\sqrt{34},\frac{44909}{249696}\sqrt{34}\right)$ &\\
& & $\left(-\frac{11}{204}\sqrt{34},-\frac{44909}{249696}\sqrt{34}\right)$ &\\
\hline
 \multirow{2}{*}{$(\ol{(0,1)},\ol{(0,1)})$} & \multirow{2}{*}2 &  $(0,1)$ & \\
 & & $x(z_1)=x(z_2)=\pm (2\cdot 3 + 2\cdot 3^3 + O(3^5)) $ & $\ord_p(x(\varphi_2(z)))=-2$\\ 
\hline
$(\ol{(0,2)},\ol{(0,1)})$ & $2$ & $x(z_1)= \pm( 2\cdot 3 + 3^2 + 2\cdot 3^3 +O(3^5))=  -x(z_2)$& $\ord_p(x(\varphi_2(z)))=-2$ \\
\hline
$(\ol{(1,0)},\ol{(1,0)})$ &$2$ & $\left(-\frac{1}{2},\pm \frac{9}{8}\right)$ & \\
\hline
$(\ol{(1,0)},\ol{(2,0)})$ &$2$ & $x(z_1)=1 + 3 + 3^3 + O(3^5)= -x(z_2)$, $y(z_1)=-y(z_2)$ & \ref{O2} with $n_1=\pm 1$, $n_2=3$ \\
\hline 
\end{tabular}}
\caption{Computation of $B^1\cup B^2$ for Example \ref{Wetherell}.}
\label{table:weth_quad}
\end{center}
\end{table}

Besides recovering representatives of the points in
\begin{equation*}
X(\mathbb{Q})\cup\left\{\left(\frac{11}{204}\sqrt{34},\pm \frac{44909}{249696}\sqrt{34}\right),\left(-\frac{11}{204}\sqrt{34},\pm \frac{44909}{249696}\sqrt{34}\right) \right\}
\end{equation*}
up to automorphisms, our zero sets contain some points $\zvec=(z_1,z_2)\in X(\Q_p)\times X(\Q_p)$ which we do not recognize as points coming from $X(K)$. Using one of the following arguments, we can prove that each of these points is indeed not $K$-rational. Suppose that $\zvec\in \sigma(X(K))$. Then, since $E_2(K)=E_2(\Q)$, the computed approximations of $\varphi_2(z_1)$ and $\varphi_2(z_2)$ must agree. If this holds, let $\varphi_2(z)$ be either of $\varphi_2(z_1)$ and $\varphi_2(z_2)$.
\begin{enumerate}[label = (O\arabic*)]
\item\label{O1} If $\varphi_2(z)$ is in the formal group at $p$, it was shown in \cite[Example 6.3]{bia:QCKim} that we must have $\ord_p(x(\varphi_2(z)))\leq -4$.
\item\label{O2} Else, since $E_2(\Q)$ is generated by $Q_2=(0,1)$, if $\varphi_2(z)=nQ_2$ for some $n\in\Z$, then
\begin{equation*}
n\equiv n_1\colonequals\frac{\int_{\infty}^{\varphi_2(z)}\omega_0^{E_2}}{\int_{\infty}^{Q_2}\omega_0^{E_2}}\bmod{3};
\end{equation*}
on the other hand, $\overline{Q_2}\in \overline{E_2}(\F_3)$ has order $6$. Thus, if $\overline{\varphi_2(z)}=n_2\overline{Q_2}$, we also have
\begin{equation*}
n \equiv n_2 \bmod{6}. 
\end{equation*}
If $n_1\not\equiv n_2 \bmod{3}$, then $\zvec\not\in \sigma(X(K))$.
\end{enumerate}
The extra points in the residue pair $(\ol{(0,1)},\ol{(0,1)})$ are precisely the extra
  points of the computation over $\Q$ (\cite[Example 6.3]{bia:QCKim}). Also note that the
  criterion \ref{O1} could be deduced from \ref{O2}.
\end{example}

\appendix
\section{Roots of multivariate systems of $p$-adic equations}
\label{appendix:roots_multi}
Some of the algorithms of Sections \ref{S:alg_ex_int} and \ref{S:examples_rational} require us to solve systems of two $p$-adic equations in two variables. In particular, in each pair of residue disks we reduce the problem to that of finding all $(t_1,t_2)\in \mathbb{Z}_p\times \mathbb{Z}_p$ such that
\begin{equation}
\label{eq:system}
\begin{cases}
\rho_1(pt_1,pt_2) = 0\\
\rho_2(pt_1,pt_2) = 0
\end{cases}
\end{equation}
for some $\rho_1(t_1,t_2),\rho_2(t_1,t_2)\in \mathbb{Q}_p[[t_1,t_2]]$.

The power series $\rho_1$ and $\rho_2$ are convergent on $p\mathbb{Z}_p\times p\mathbb{Z}_p$ and, possibly after multiplying by a power of $p$, we may assume that $\rho_1(pt_1,pt_2),\rho_2(pt_1,pt_2)\in \mathbb{Z}_p[[t_1,t_2]]\setminus p\mathbb{Z}_p[[t_1,t_2]]$.

A careful study of the valuation of the coefficients of $\rho_1$ and $\rho_2$ allows us to conclude that if $(\alpha,\beta)\in \mathbb{Z}_p\times \mathbb{Z}_p$ is a solution  to (\ref{eq:system}), then $(\alpha \bmod{p^n},\beta \bmod{p^n})$ is a solution modulo $(p^n)$ to the system
\begin{equation}
\label{eq:systemmodpn}
\begin{cases}
\rho_1(pt_1,pt_2)+O(t_1,t_2)^{\tilde{n}} = 0\\
\rho_2(pt_1,pt_2)+O(t_1,t_2)^{\tilde{n}} = 0
\end{cases}
\end{equation}
for some integer ${\tilde{n}}$. How reversible is this process? In other words, if we find all the solutions of (\ref{eq:systemmodpn}) mod $p^n$, how do we know if each such solution \emph{lifts} to a \emph{unique} solution to (\ref{eq:system})?

In some cases, we can apply a multivariate version of Hensel's lemma (see Theorem \ref{thm:multihensel} below). Let us first introduce some notation. Let $m\in\mathbb{N}$ and, for an element $\ul{a}=(a_1,\dots,a_m)\in \mathbb{Q}_p^m$, define $\ord_p(\ul{a}) = \min_{1\leq i\leq m}\{\ord_p(a_i)\}$. Then $\ord_p$ satisfies the following properties:
\begin{enumerate}[label = (\alph*)]
\item For any $k\in\mathbb{Q}_p$, $\ord_p(k\ul{a})=\ord_p(k)+\ord_p(\ul{a}).$
\item\label{prop:nonarch} $\ord_p(\ul{a}+\ul{b})\geq \min\{\ord_p(\ul{a}),\ord_p(\ul{b})\}$ with equality if $\ord_p(\ul{a})\neq\ord_p(\ul{b})$.
\end{enumerate}

 For $\ul{f}=(f_1,\dots,f_m)\in\Z_p[[x_1,\dots,x_m]]^m$ and $\ul{x}=(x_1,\dots,x_m)$, define the Jacobian matrix
\begin{equation*}
J_{\ul{f}}(\ul{x})\colonequals  \begin{bmatrix} 
    \frac{\partial{f_1}}{\partial{x_1}} & \frac{\partial{f_1}}{\partial{x_2}} & \dots & \frac{\partial{f_1}}{\partial{x_m}}\\
    \vdots & \ddots &\vdots \\
    \frac{\partial{f_m}}{\partial{x_1}} &\frac{\partial{f_m}}{\partial{x_2}}&    \dots     & \frac{\partial{f_m}}{\partial{x_m}}
    \end{bmatrix}(\ul{x}).
\end{equation*}

Denote by $\Z_p\langle x_1,\dots, x_m \rangle$ the subring of $\Z_p[[x_1,\dots, x_m]]$ consisting of those power series whose coefficients tend to $0$ as the degree of the corresponding monomials tends to infinity.
\begin{theorem}[Multivariate Hensel's lemma \hspace{1sp}{\cite[Theorem 4.1]{conrad:multihensel}}]
\label{thm:multihensel}
Let\\
$\ul{f}=(f_1,\dots,f_m)\in \Z_p\langle x_1,\dots, x_m \rangle^m$. Assume that $\ul{a}\in \mathbb{Z}_p^m$ satisfies
\[
\ord_p(\ul{f}(\ul{a}))>2 \ord_p(\det(J_{\ul{f}}(\ul{a}))).
\]
 Then there is a unique $\ul{\alpha}\in\mathbb{Z}_p^m$ such that $\ul{f}(\ul{\alpha})=\ul{0}$ and $\ord_p(\ul{\alpha}-\ul{a})>\ord_p(\det(J_{\ul{f}}(\ul{a})))$. Furthermore, $\ord_p(\ul{\alpha}-\ul{a})\geq \ord_p(\ul{f}(\ul{a}))-\ord_p(\det(J_{\ul{f}}(\ul{a})))$.
\end{theorem}
Even though \cite{conrad:multihensel} is not the only reference for Theorem \ref{thm:multihensel} (see for example also \cite[Theorem 23]{Kuhlmann}), the proof given by Conrad in \cite{conrad:multihensel} has the advantage of being constructive: it uses Newton's method to explicitly find approximations of $\ul{\alpha}$ and it shows how fast these converge to $\ul{\alpha}$. In particular, with the notation of the theorem, let $\ul{a}_1=\ul{a}$ and define, for $N\geq 1$,
\begin{equation*}
\ul{a_{N+1}}= \ul{a_{N}} -\ul{f}(\ul{a_{N}})\cdot (J_{\ul{f}}(\ul{a_N})^T)^{-1}\,.
\end{equation*}
By \cite[(3.9)]{conrad:multihensel}, we have the inequality
\begin{equation*}
\ord_p(\ul{a_{N+1}}-\ul{a_N})\geq \ord_p(\det(J_{\ul{f}}(\ul{a})))+2^{N-1}(\ord_p(\ul{f}(\ul{a}))-2\ord_p(\det(J_{\ul{f}}(\ul{a}))));
\end{equation*}
combining this with Property \ref{prop:nonarch} and with the assumption that $$\ord_p(\ul{f}(\ul{a}))-2\ord_p(\det(J_{\ul{f}}(\ul{a})))>0,$$ we obtain
\begin{equation*}
\ord_p(\ul{a_{M}}-\ul{a_N})\geq \ord_p(\det(J_{\ul{f}}(\ul{a})))+2^{N-1}(\ord_p(\ul{f}(\ul{a}))-2\ord_p(\det(J_{\ul{f}}(\ul{a}))))
\end{equation*}
for all $M>N$. Thus, taking limits as $M\to \infty$,
\begin{equation*}
\ord_p(\ul{\alpha}-\ul{a_N})\geq \ord_p(\det(J_{\ul{f}}(\ul{a})))+2^{N-1}(\ord_p(\ul{f}(\ul{a}))-2\ord_p(\det(J_{\ul{f}}(\ul{a})))).
\end{equation*}

We sketch a possible strategy to compute roots of systems of equations of the form (\ref{eq:system}), omitting technicalities on precision. 
This immediately generalizes to the case where (\ref{eq:system}) is replaced by a system of $m$ equations in $m$ variables for arbitrary $m\geq 2$.
\begin{algorithm}\label{alg} Computing roots of (\ref{eq:system})
\begin{itemize}
\item[Input:] A system $\ul{\rho}=(\rho_1,\rho_2)$ of the form (\ref{eq:system}).
\item[Output:] $L$ -- list of roots modulo $p^n$ of $\ul{\rho}$.
\end{itemize}
\begin{enumerate}
\item Reduce to (\ref{eq:systemmodpn}).
\item Compute $R_p =\{(t_1,t_2)\in\{0,\dots,p-1\}^2:\ul{\rho}(pt_1,pt_2)\equiv 0 \bmod{p}\} $.
\item Let $\ul{a}\in R_p$. If $\ord_p(\det(J_{\ul{\rho}}(\ul{a})))=0$, by Theorem \ref{thm:multihensel} there exists a unique root $\ul{\alpha}\in \mathbb{Z}_p^2$ of (\ref{eq:system}) such that $\ord_p(\ul{\alpha}-\ul{a})>0$. Compute $\ul{\alpha}$ to the desired precision using the discussion following Theorem \ref{thm:multihensel} and append it to $L$.
\item\label{en:fourth_opt} Let $\ul{a}\in R_p$ such that $\ord_p(\det(J_{\ul{\rho}}(\ul{a})))\neq 0$. Fix an $r\geq 3$, $r\leq n$ and naively find {\small $R_{p^r}(\ul{a})=\{(t_1,t_2)\in\{0,\dots,p^r-1\}^2:\ul{\rho}(pt_1,pt_2)\equiv 0 \bmod{p^r}\text{ and } (t_1,t_2)\equiv \ul{a}\bmod{p}\}$}. For each $\ul{b}\in R_{p^r}(\ul{a})$ do the following:
\begin{enumerate}[label =(\roman*)]
\item\label{en:hensel} if $2\ord_p(\det(J_{\ul{\rho}}(\ul{b})))< r$, compute using Theorem \ref{thm:multihensel} an approximation of the unique root $\ul{\beta}\in\mathbb{Z}_p^2$ such that $\ord_p(\ul{\beta}-\ul{b})>\ord_p(\det(J_{\ul{\rho}}(\ul{b})))$ and append it to $L$.  
  Remove from $R_{p^r}(\ul{a})$ all $\ul{b}^{\prime}$ such that $\ord_p(\ul{b}^{\prime}-\ul{b})>\ord_p(\det(J_{\ul{\rho}}(\ul{b})))$.
\item\label{en:nothensel} if $2\ord_p(\det(J_{\ul{\rho}}(\ul{b})))\geq  r$, then if $r=n$ append $\ul{b}$ to $L$; if $r<n$, choose some $s>r$, $s\leq n$, naively find  
$R_{p^s}(\ul{b})=\{(t_1,t_2)\in\{0,\dots,p^s-1\}^2:\ul{\rho}(pt_1,pt_2)\equiv 0 \bmod{p^s}\text{ and } (t_1,t_2)\equiv \ul{b}\bmod{p^r}\}$ and for each $\ul{c}\in R_{p^s}(\ul{b})$ repeat \ref{en:hensel}, \ref{en:nothensel} with the appropriate modification in notation. 
\end{enumerate}
\end{enumerate}
\end{algorithm}

\begin{remark}
\label{rmk:auto_double_roots}
If the algorithm never appends to $L$ in \ref{en:fourth_opt}\thinspace \ref{en:nothensel}, each element in $L$ is the
  approximation of a unique root of (\ref{eq:system}) and, conversely, each root of \eqref{eq:system} reduces to an element of $L$. Otherwise, we can try to increase
  $n$. However, there are some residue disks where the systems that we consider provably have roots at which the determinant of the Jacobian matrix also vanishes, so Theorem \ref{thm:multihensel} is not applicable.  Indeed, let $X$ be as in Section~\ref{sec:rat}. Then $X$ admits the automorphisms $\theta\colon(x,y)\mapsto (\pm x,\pm y)$. Composing $\varphi_k$ with $\theta$ we obtain either the identity or multiplication by $-1$ on $E_k$. Since local heights are even, it follows that if a function $\rho$ comes from a height function as in Theorem \ref{thm:bielliptic}, then all its partial derivatives vanish\footnote{E.g.\ for the pair of disks containing a point fixed by $(x,y)\mapsto(-x,y)$, we can set $t_1=x_1$, $t_2=x_2$ as uniformizers. Then $\rho(t_1,t_2)=\rho(-t_1,-t_2)$ giving vanishing of the partial derivatives at $t_1=t_2=0$.} at a $K$-rational point fixed by a nontrivial automorphism $\theta$. A function coming from relations of the linear functionals $f_i$ on $E_k$ is odd and will thus have non-simple zeros in the above sense at those points $(x,y)$ which are fixed by a nontrivial automorphism $\theta$ that projects to the identity on $E_k$. 
\end{remark}

We now explain how to deal with the disks of Remark \ref{rmk:auto_double_roots} where Hensel is not applicable, for the explicit algorithms that we described in Section \ref{S:examples_rational}. Suppose first that we are in the situation of \S \ref{sec:rank_4_rational}, of which we retain the notation. Let $H$ be the subgroup of the automorphism group of $X$ generated by the hyperelliptic involution and the bielliptic automorphism $(x,y)\mapsto (-x,y)$.
\begin{lemma}
\label{lemma:double_roots}
Let $K$ and $X/\Q$ be as in \S \ref{sec:rank_4_rational}. Let $P\in X(K)$ such that $\theta(P) = P$ for a nontrivial automorphism $\theta\in H$; if $y(P)=0$, further assume that $x(P)^2\in \Q$. Let $k\in\{1,2\}$ such that $P\in X^{(k)}(K)$ and let $w\in T^{\anti,k}$ such that $\rho^{\anti,k}(\sigma(P)) = w$. 
Then we can choose local coordinates $t_1$ and $t_2$ for $\sigma_1(P)$ and $\sigma_2(P)$, such that, in the residue pair $\overline{\sigma_1(P)}\times \overline{\sigma_2(P)}\in \overline{X}(\mathbb{F}_p)\times \overline{X}(\mathbb{F}_p)$, we have
\begin{equation*}
\rho^{\anti,k}(t_1,t_2)-w = \sum_{i=1}^{\infty}c_i(t_1^{2i}-t_2^{2i}) = (t_1^2-t_2^2)\sum_{i=1}^{\infty}c_i\biggl(\sum_{j=0}^{i-1}t_1^{2j}t_2^{2i-2-2j}\biggr)
\end{equation*}
for some $c_i\in \Q_p$.
\end{lemma}

\begin{proof}
Write a point in $X(\Q_p)\times X(\Q_p)$ as $((x_1,y_1),(x_2,y_2))$. We fix the following choices for $t_1$ and $t_2$. If $\theta(x,y)=(-x,y)$, let $t_i =  x_i$; if $\theta(x,y)=(x,-y)$, let $t_i =  y_i$; finally, if $P=\infty^{\pm}$, let $t_i = 1/x_i$. 
As observed in Remark \ref{rmk:auto_double_roots}, we have 
\[\rho^{\anti,k}((x_1,y_1),(x_2,y_2)) = \rho^{\anti,k}(\theta(x_1,y_1),\theta(x_2,y_2));\]
our choice of local coordinates then yields
\begin{equation*}
\rho^{\anti,k}(t_1,t_2) =  \rho^{\anti,k}(-t_1,-t_2).
\end{equation*}
Since by \S\ref{sec:ell_imag} we have $\alpha_{01}^{\anti, E_j}=0$ for each $j\in \{1,2\}$, this implies that there exist some power series $f_1$ and $f_2$ such that
\begin{equation*}
\rho^{\anti,k}(t_1,t_2) =  f_1(t_1^2)-f_2(t_2^2) + 2h_p^{\anti,E_k}(Q_k).
\end{equation*}
In view of our assumptions, we also have that $x_1(t_1)^2=x_2(t_2)^2|_{t_2 = t_1}$ and thus that $\varphi_j(x_1(t_1),y_1(t_1))=\pm \varphi_j(x_2(t_2),y_2(t_2))|_{t_2 = t_1}$. 
Since $\alpha_{00}^{\anti, E_j}=-\alpha_{11}^{\anti,E_j}$ by \S\ref{sec:ell_imag}, we conclude that $f_1=f_2$. 
Finally, the lemma follows by definition of $w$.
\end{proof}
\begin{remark}
Since $h_p^{\anti,E_k}(Q_k)=0$, we must have $w=0$.
\end{remark}

In the notation of Lemma \ref{lemma:double_roots}, the points $(z_1,z_2)\in X^{(k)}(K\otimes \Q_p)$ reducing to $\overline{\sigma_1(P)}\times \overline{\sigma_2(P)}$ and such that $\rho^{\anti,k}(z_1,z_2)=w$ satisfy one of the following:
\begin{enumerate}[label=(\roman*)]
\item\label{en:t1t2} $t_1(z_1)=t_2(z_2)$ and $\rho^{\cyc,k}(t_1(z_1),t_1(z_1))\in T^{k,\cyc}$;
\item\label{en:t1mt2} $t_1(z_1)=-t_2(z_2)$ and $\rho^{\cyc,k}(t_1(z_1),-t_1(z_1))\in T^{k,\cyc}$;
\item\label{en:remaining} 
$\begin{cases}
\sum_{i=1}^{\infty}c_i\left(\sum_{j=0}^{i-1}t_1(z_1)^{2j}t_2(z_2)^{2i-2-2j}\right)=0\\
\rho^{\cyc,k}(t_1(z_1),t_2(z_2))\in T^{k,\cyc}.
\end{cases}$
\end{enumerate}
If $c_1\neq 0$, then $(0,0)$ is not a root of \ref{en:remaining}, so in Algorithm \ref{alg} we can hope for no root to be appended in \ref{en:fourth_opt}\thinspace \ref{en:nothensel}.
We also note that the one-variable power series of \ref{en:t1t2} and \ref{en:t1mt2} are power series in $t_1^2$.

In the setting of \S \ref{sec:examples_bielliptic_real}, we follow a similar strategy (and adopt some of the notation of the imaginary case). Recall that the second function $\rho_2$ is given by
\begin{equation*}
\rho_2(z_1,z_2) =\int_{\infty}^{\varphi_2(z_1)}\omega_0^{E_2}-b\int_{\infty}^{\varphi_2(z_2)}\omega_0^{E_2},
\end{equation*}
where $b=1$ or $b=-1$, depending on the rank of $E_2$ over $\Q$. We assume here that $b=1$ as in Example \ref{Wetherell}. 
 Let $P\in X^{(k)}(K)$ be fixed by $\theta\in H\setminus\{\mathrm{id}\}$ and let $t_1$ and $t_2$ be local coordinates chosen in the same way as in the proof of Lemma \ref{lemma:double_roots}. Then, in the residue disk of $\sigma(P)$, we have
\begin{equation*}
\rho_2(-t_1,-t_2)= \begin{cases}
\rho_2(t_1,t_2) & \text{if } \varphi_2\circ\theta = \mathrm{id},\\
-\rho_2(t_1,t_2) & \text{if } \varphi_2\circ\theta = -\mathrm{id}.
\end{cases}
\end{equation*}
Assume for simplicity that $P\in X(\Q)$, so the local expansion of $\int_{\infty}^{\varphi_2(z_1)}\omega_0^{E_2}$ evaluated at $t_2$ gives the local expansion of $\int_{\infty}^{\varphi_2(z_2)}\omega_0^{E_2}$. We conclude that
\begin{equation*}
\rho_2(t_1,t_2)= \begin{cases}
(t_1^2-t_2^2)\tilde{\rho_2}(t_1,t_2) & \text{if } \varphi_2\circ\theta = \mathrm{id},\\
(t_1-t_2)\tilde{\rho_2}(t_1,t_2) & \text{if } \varphi_2\circ\theta = -\mathrm{id}
,\end{cases}
\end{equation*}
for some $\tilde{\rho_2}(t_1,t_2)\in \Q_p[[t_1,t_2]]$.

\bibliographystyle{amsalpha}
\bibliography{total}
\end{document}